\documentclass[11pt]{article}
\def\thetitle{Greedy maximal independent sets via local limits}

\usepackage{amsmath,amssymb}

\usepackage[usenames,dvipsnames,svgnames,table]{xcolor}
\definecolor{CombinatoricaAqua}{HTML}{00698C}
\definecolor{CombinatoricaBlue}{HTML}{3A3293}
\definecolor{CombinatoricaBrown}{HTML}{66220C}
\definecolor{CombinatoricaRed}{HTML}{DF2A27}
\definecolor{HarvardCrimson}{rgb}{0.6471, 0.1098, 0.1882}

\makeatletter
\let\reftagform@=\tagform@
\def\tagform@#1{\maketag@@@
	{(\ignorespaces\textcolor{CombinatoricaBrown}{#1}\unskip\@@italiccorr)}}
\renewcommand{\eqref}[1]{\textup{\reftagform@{\ref{#1}}}}
\makeatother

\usepackage[backref=page]{hyperref}
\hypersetup{
	unicode,
	pdfencoding=auto,
	pdfauthor={Michael Krivelevich, Tam\'as M\'esz\'aros, Peleg Michaeli, Clara Shikhelman},
	pdftitle={\thetitle},
	pdfsubject={},
	pdfkeywords={},
	colorlinks=true,
	citecolor=CombinatoricaBlue,
	linkcolor=CombinatoricaAqua,
	anchorcolor=CombinatoricaBrown,
	urlcolor=HarvardCrimson}

\usepackage{amsthm}
\usepackage{thmtools}
\usepackage{bbm}
\usepackage[capitalize]{cleveref}
\Crefname{fact}{Fact}{Facts}
\Crefname{claim}{Claim}{Claims}
\Crefname{assumption}{Assumption}{Assumptions}


\declaretheoremstyle[
spaceabove=\topsep, spacebelow=\topsep,
headfont=\color{CombinatoricaBrown}\normalfont\bfseries,
bodyfont=\itshape,
]{thm}
\declaretheoremstyle[
spaceabove=\topsep, spacebelow=\topsep,
headfont=\color{CombinatoricaBrown}\normalfont\bfseries,
bodyfont=\normalfont,
]{dfn}
\declaretheoremstyle[
spaceabove=0.5\topsep, spacebelow=0.5\topsep,
headfont=\color{CombinatoricaBrown}\normalfont\bfseries,
bodyfont=\normalfont,
]{rmk}

\declaretheorem[style=thm,parent=section]{theorem}
\declaretheorem[style=thm,sibling=theorem]{lemma}

\declaretheorem[style=thm,sibling=theorem]{claim}
\declaretheorem[style=thm,sibling=theorem]{proposition}

\declaretheorem[style=definition,numbered=no]{acknowledgement}

\usepackage[nobysame,msc-links,non-sorted-cites]{amsrefs}

\BibSpec{book}{%
	+{}  {\PrintPrimary}                {transition}
	+{,} { \textbf}                     {title} 
	+{.} { }                            {part}
	+{:} { \textit}                     {subtitle}
	+{,} { \PrintEdition}               {edition}
	+{}  { \PrintEditorsB}              {editor}
	+{,} { \PrintTranslatorsC}          {translator}
	+{,} { \PrintContributions}         {contribution}
	+{,} { }                            {series}
	+{,} { \voltext}                    {volume}
	+{,} { }                            {publisher}
	+{,} { }                            {organization}
	+{,} { }                            {address}
	+{,} { \PrintDateB}                 {date}
	+{,} { }                            {status}
	+{}  { \parenthesize}               {language}
	+{}  { \PrintTranslation}           {translation}
	+{;} { \PrintReprint}               {reprint}
	+{.} { }                            {note}
	+{.} {}                             {transition}
	+{}  {\SentenceSpace \PrintReviews} {review}
}

\makeatletter
\renewcommand{\PrintNames@a}[4]{%
	\PrintSeries{\name}
	{#1}
	{}{ and \set@othername}
	{,}{ \set@othername}
	{}{ and \set@othername}
	{#2}{#4}{#3}%
}
\makeatother

\makeatletter
\def\mathcolor#1#{\@mathcolor{#1}}
\def\@mathcolor#1#2#3{%
	\protect\leavevmode
	\begingroup
	\color#1{#2}#3%
	\endgroup
}
\makeatother
\definecolor{Red}{rgb}{0.618,0,0}
\definecolor{Blue}{rgb}{0,0,1}
\definecolor{Green}{rgb}{0,0.298,0}

\usepackage{sectsty}
\sectionfont{\color{CombinatoricaBrown}}
\subsectionfont{\color{CombinatoricaBrown}}
\subsubsectionfont{\color{CombinatoricaBrown}}

\usepackage{soul}
\soulregister\ref7

\usepackage[textsize=tiny,backgroundcolor=black!10]{todonotes}

\usepackage{pifont}

\usepackage[a4paper]{geometry}
\usepackage[shortlabels]{enumitem}

\title{\thetitle}
\author{
	Michael Krivelevich%
	\thanks{School of Mathematical Sciences, Tel Aviv University, Tel Aviv 6997801, Israel. E-mail:
		\href{mailto:krivelev@tauex.tau.ac.il}
		{\tt krivelev@tauex.tau.ac.il}.
    Research supported in part by USA-Israel BSF grant~2018267, and by ISF grant 1261/17.}
	\and Tam\'as M\'esz\'aros%
	\thanks{Fachbereich Mathematik und Informatik, Kombinatorik und Graphentheorie, Freie Universit\"at Berlin, Arnimallee 3, 14195 Berlin, Germany. E-mail:
	    \href{mailto:tmeszaros87@gmail.com}{\tt tmeszaros87@gmail.com}.
	    Research supported by the Berlin Mathematics Research Center MATH+.}
	\and Peleg Michaeli%
	\thanks{%
    School of Mathematical Sciences, Tel Aviv University, Tel Aviv 6997801, Israel,
    and Department of Mathematical Sciences,  Carnegie Mellon University, Pittsburgh, PA 15213, USA.
    E-mail: \href{mailto:pelegm@cmu.edu}{\tt pelegm@cmu.edu}.}
	\and Clara Shikhelman%
	\thanks{Chaincode Labs, 450 Lexington Avenue, New York, NY 10017, USA. E-mail:
		\href{mailto:clara.shikhelman@gmail.com}
		{\tt clara.shikhelman@gmail.com}.}
}

\usepackage{mathtools}
\usepackage{mleftright}
\mleftright

\newcommand{\NN}{\mathbb{N}}
\newcommand{\ZZ}{\mathbb{Z}}
\newcommand{\cE}{\mathcal{E}}
\newcommand{\cN}{\mathcal{N}}
\newcommand{\cQ}{\mathcal{Q}}
\newcommand{\sA}{\mathsf{A}}
\newcommand{\defn}[1]{\textbf{#1}}
\newcommand{\term}[1]{\textit{#1}}
\newcommand{\eps}{\varepsilon}
\newcommand{\T}{\mathsf{T}}
\newcommand{\F}{\mathsf{F}}

\newcommand{\sm}{\smallsetminus}

\DeclareMathOperator{\dist}{dist}

\newcommand{\vect}{\mathbf}

\newcommand{\pr}[0]{\mathbb{P}}
\newcommand{\E}[0]{\mathbb{E}}
\DeclareMathOperator{\Var}{Var}
\DeclareMathOperator{\cov}{cov}
\newcommand{\bin}{\mathsf{Bin}}
\newcommand{\pois}{\mathsf{Pois}}
\newcommand{\iid}{\textbf{iid}}
\newcommand{\fiid}{\textbf{fiid}}
\newcommand{\whp}{\textbf{whp}}
\newcommand{\pgf}{\textbf{pgf}}
\newcommand{\sfpg}{\textbf{sfpg}}

\DeclareMathOperator{\rad}{rad}
\DeclareMathOperator{\rng}{rng}

\newcommand{\gis}{\mathbf{I}}
\newcommand{\gic}{\mathfrak{i}}
\newcommand{\gir}{\iota}
\newcommand{\egic}{\bar{\mathfrak{i}}}
\newcommand{\egir}{\bar{\iota}}
\newcommand{\dloc}{d_{\mathrm{loc}}}
\newcommand{\Gdot}{\mathcal{G}_{\bullet}}
\newcommand{\past}[1]{\mathcal{P}_{#1}}

\newcommand{\toloc}{\xrightarrow{\mathrm{loc}}}
\newcommand{\ary}[1]{\mathsf{T}_{#1}}
\newcommand{\istar}[1]{\mathcal{S}_{#1}}
\newcommand{\gw}[1]{\mathcal{T}_{#1}}
\newcommand{\reg}[1]{\mathbb{T}_{#1}}
\newcommand{\sbgw}[1]{\hat{\mathcal{T}}_{#1}}
\DeclareMathOperator{\KC}{KC}

\usepackage{comment}
\specialcomment{todelete}{\begingroup\color{gray}}{\endgroup}

\newcommand{\concat}{%
  \mathbin{\raisebox{1ex}{\scalebox{.7}{$\frown$}}}%
}

\usepackage{subcaption}
\usepackage{tikz}
\usetikzlibrary{calc,trees}
\usepackage{booktabs}

\usepackage{tabto}

\begin{document}
\maketitle

\begin{abstract}
  The random greedy algorithm for finding a maximal independent set in a graph
  constructs a maximal independent set by inspecting the graph's vertices in a
  random order,
  adding the current vertex to the independent set if it is not
  adjacent to any previously added vertex.
  In this paper,
  we present a general framework for computing the asymptotic density of the
  random greedy independent set
  for
  sequences of (possibly random) graphs
  by employing a notion of
  local convergence.
  We use this framework to give straightforward proofs
  for results on previously studied families of graphs,
  like paths and binomial random graphs,
  and to study new ones,
  like random trees and sparse random planar graphs.
  We conclude by analysing the random greedy algorithm more closely when the base graph is a tree.
\end{abstract}

\section{Introduction}\label{sec:intro}
An \defn{independent set} in a graph is a set of vertices, no two of which are adjacent.
The problem of finding large independent sets is fundamental in computer science, with many real-world applications.
Computing the size of a \emph{maximum} independent set (known as the \term{independence number} of a graph) is known to be NP-hard on general graphs~\cite{Kar72},
and is even hard to approximate~\cite{FGLSS96}.
A natural way to try to efficiently produce a large independent set in an input graph $G$ is to output a \emph{maximal} independent set (MIS), namely, an independent set to which no other vertex can be added without destroying its property of being independent.
While in principle a poorly chosen MIS can be very small (like, say, the star centre in a star), one might hope that quite a few of the maximal independent sets will have a size comparable in some quantitative sense to the independence number of $G$.

This paper\footnote{%
This is an extended and revised version of a conference version presented at the 31st International Conference on Probabilistic, Combinatorial and Asymptotic Methods for the Analysis of Algorithms (AofA2020)~\cite{KMMS20}.}
studies the \defn{random greedy algorithm} for producing an MIS, which is defined as follows.
Given an input graph $G$, the algorithm first orders its vertices uniformly at random and then constructs an independent set $\gis=\gis(G)$ by considering each of the vertices one by one in order, adding it to $\gis$ if the resulting set does not span an edge.
(Note that the set $\gis$ is, in fact, the set of vertices coloured in the first colour in a random greedy proper colouring of $G$.)
A basic quantity to study, which turns out to have numerous applications, is the proportion of the yielded independent set $|\gis|/|V(G)|$ (which we call the \defn{greedy independence ratio}).
In particular, it is of interest to study the asymptotic behaviour of this quantity for natural (random) graph sequences.

Due to its simplicity, the random greedy algorithm has been studied extensively by various authors in different fields, ranging from combinatorics~\cite{Wor95}, probability~\cite{RV17} and computer science~\cite{FN18}
to chemistry~\cite{Flo39}.
As early as 1931, this model was studied by chemists under the name \emph{random sequential adsorption} (RSA), focusing primarily on $d$-dimensional grids.
The $1$-dimensional case (a path) was solved by Flory~\cite{Flo39} (see also
\cite{Pag59}), who showed that the expected greedy independence ratio tends to
$\zeta_2=(1-e^{-2})/2$ as the path length tends to infinity.

A continuous analogue, in which ``cars'' of unit length ``park'' at random free locations on the interval $[0,X]$, was introduced (and solved) by R\'enyi~\cite{Ren58}, under the name \term{car-parking process}.
The limiting density, as $X$ tends to infinity, is called \term{R\'enyi's parking constant}, and $\zeta_2$ may be considered as its discrete counterpart (see, e.g.,~\cite{MathC}).
Following this terminology, the final state of the car-parking process is often called the \emph{jamming limit} of the graph, and the density of this state is called the \emph{jamming constant}.
For dimension $2$, Pal\'asti~\cite{Pal60} conjectured, in the continuous case (where unit square ``cars'' park in a larger square), that the limiting density is R\'enyi's parking constant squared.
This conjecture may be carried over to the discrete case, but to the best of our knowledge, in both cases, it remains open.
For further details, see~\cite{MathC} (see also~\cite{Eva93} for an extensive survey on RSA models, and~\cite{DLM16} for generalisations of the RSA model).

In combinatorics, the greedy algorithm for finding an MIS was analysed in order to give a lower bound on the (usually asymptotic) typical independence number of (random) graphs.\footnote{%
For this purpose, more sophisticated local (and non-local) algorithms have been analysed. Nevertheless, as we mention later, the random greedy algorithm is, perhaps surprisingly, at least as good as any other local algorithm for various random graph models.
In fact, in many problems, the random greedy algorithm is essentially the best known efficient algorithm available.}
The asymptotic greedy independence ratio of binomial random graphs was studied by McDiarmid~\cite{McD84}
(but see also~\cites{GM75,BE76}; for large deviation estimates, see~\cites{BGJM22,Kol22}).
The asymptotic greedy independence ratio of random regular graphs was studied by Wormald~\cite{Wor95}, who used the so-called \term{differential equation method} (see \cite{Wor99} for a comprehensive survey; see also \cite{War19+} for a short proof of Wormald's result).
His result was further extended in~\cite{LW07} to any sequence of regular graphs with growing girth
(see also \cites{HW16,HW18} for similar extensions for more sophisticated algorithms%
\footnote{For example, there is a series of works obtaining ever better lower bounds for the independence ratios in (random/high-girth) $3$- and $4$-regular graphs, using local algorithms;
see, e.g.,~\cites{DZ09,CGHV15,HW16,Cso16+}.}%
).
The case of uniform random graphs with given degree sequences was studied (independently) in~\cite{BJM17} and~\cite{BJL17}.

Apart from being basic to combinatorial optimisation, the random greedy algorithm for producing an MIS is of pure theoretic interest:
it is a simple and natural stochastic process which, in its most general form, emulates many previously studied processes.
One such model is the \term{randomised greedy matching} (see~\cite{DF91}), which can be defined as the random greedy MIS on the corresponding line graph.
This model was studied by Dyer and Frieze~\cite{DF91} for general graphs and by Dyer, Frieze and Pittel~\cite{DFP93} for sparse uniform random graphs.
The distribution of $\gis$ also appears naturally as a marginal of a random greedy colouring of the graph.

In a more general setting, where the random greedy algorithm runs on a \term{hypergraph}, the model recovers, in particular, the \emph{triangle-free process} (or, more generally, the \emph{$H$-free process}).
In this process, which was first introduced in~\cite{ESW95}, we begin with the empty graph, and at each step, add a random edge as long as it does not create a copy of a triangle (or of $H$).
To recover this process, we take the hypergraph whose vertices are the edges of the complete graph and whose hyperedges are the triples of edges that span a triangle (or $k$-sets of edges that form a copy of $H$, if $H$ has $k$ edges).
Bohman's key result~\cite{Boh09} is that for this hypergraph, $|\gis|$ is with high probability (\whp{})\footnote{That is, with probability tending to $1$ as $n$ tends to infinity.} $\Theta(n^{3/2}\sqrt{\ln{n}})$, where $n$ is the number of vertices.
Bohman and Keevash~\cite{BK13} and Fiz Pontiveros, Griffiths and Morris~\cite{FPGM20} later found the exact asymptotics.
Similar results were obtained for the complete graph on four vertices by Warnke~\cite{War14K4} and for cycles independently by Picollelli~\cite{Pic14} and by Warnke~\cite{War14Cl}.
For a discussion about the general setting, see~\cite{BB16}.
An additional celebrated model that can be emulated by a random greedy MIS on a (nonuniform) hypergraph is the \term{minimum spanning tree} (MST).

Consider the following alternative but equivalent definition of the model.
Assign an independent uniform \emph{label} from $[0,1]$ to each vertex of the graph, and consider it as the \emph{arrival time} of a particle at that vertex.
All vertices are initially vacant, and a vertex becomes occupied at the time denoted by its label if and only if all of its neighbours are still vacant at that time.
Clearly, we do not need to worry that two particles will arrive at the same time.
The set of occupied vertices at time $1$ is exactly the greedy MIS.
We may think of the resulting MIS as a \term{factor of iid} (\fiid{})\footnote{%
The letters \iid{} abbreviate \term{independent and identically distributed}.}, %
meaning, informally, that there exists a ``local'' rule, unaware of the ``identity'' of a given vertex, that determines whether that vertex is occupied.
It was conjectured (formally by Hatami, Lov\'asz and Szegedy~\cite{HLS14}, but see discussion in~\cite{GS17}) that, using a proper rule, \fiid{}
can produce an asymptotically maximum independent set in random regular graphs of high degree.
However, this was disproved by Gamarnik and Sudan~\cite{GS17}.
In fact, they showed  that this kind of local algorithms has a uniformly limited
power for a sufficiently large degree, and later Rahman and Vir\'ag~\cite{RV17}
showed that the density of \fiid{} independent sets in regular trees and Poisson Galton--Watson trees with a large average degree, is asymptotically at most {\em half}-optimal, concluding (after projecting to random regular graphs or to binomial random graphs) that local algorithms cannot achieve better.
In particular, this implies that the random greedy algorithm is, asymptotically, at least as good as any other local algorithm.

In general graph sequences, however, local algorithms may perform arbitrarily close to optimal.
A trivial example is the set of stars, where the greedy algorithm typically
performs perfectly.
A less trivial example is that of uniform random trees.
The expected independence ratio of a uniform random tree is the unique solution of the equation $x=e^{-x}$ (see~\cite{MM73}), which is approximately $0.5671...$, while the greedy algorithm yields an independent set of expected density $1/2$ as we will see in \cref{sec:sbgw}.

Finally, we note that the following parallel/distributed algorithm gives a further way to look at the maximal independent set generated by the greedy algorithm.
After (randomly) ordering the vertices, we add to $\gis$ all the \term{sinks}, namely, the vertices which appear before their neighbours in the order, and then remove them and their neighbours from the graph.
We repeat these steps until the graph is empty.
Formulated this way, the algorithm is straightforward to implement and requires only local communication between the nodes. Also, conditioning on the initial random ordering, it is deterministic, a property that appears to be important (see, e.g.,~\cite{BFGS12}).
A central question of interest is the number of rounds it takes the algorithm to terminate.
In~\cite{FN18} it was shown that it terminates in $O(\log{n})$ steps \whp{} on any $n$-vertex graph, and that this is tight.
Thus, even though these algorithms may be suboptimal, they are strikingly simple and surprisingly efficient.

\subsection{Our results}
The goal of this paper is to present a unified, comprehensive, and easy-to-apply framework for analysing the random greedy independence ratio.
Indeed, several previous results (e.g., results from~\cites{Flo39,Pag59,McD84,Wor95,LW07,NV21}), as well as new ones, can be derived as special cases of theorems that we present below.
The general approach is to study a suitable limiting object, typically a random rooted infinite graph, which captures the local view of a typical vertex, and to calculate the probability that its root appears in a random independent set in this graph, which is created according to some natural ``local'' rule, to be described later.
We show that this probability approximates the expected greedy independence ratio and give tools to calculate that probability precisely in many cases of interest.
This change of view from a sequence of graphs with varying underlying probability spaces to an infinite object with a fixed underlying probability space is known in the literature as the \term{objective method}~\cite{AS04}.

Let us formulate this more precisely.
A \defn{random labelling} of a (possibly infinite, possibly random) graph $G=(V,E)$
is a process $\sigma=(\sigma_v)_{v\in V}$ consisting of \iid{} random variables $\sigma_v$,
each distributed uniformly in $[0,1]$.
If $G$ is finite, we let $\gis_\sigma(G)$ denote the random greedy maximal independent set of $G$ obtained by the ordering induced by $\sigma$
(note that it is measurable w.r.t.\ $G$ and $\sigma$).
We also let $\gir(G)$ denote the density of $\gis_\sigma(G)$ and $\egir(G)$ denote its expectation (taken over the distribution of $G$ and over the labelling $\sigma$).
The past of a vertex $v$, denoted $\past{v}$, is the (random) set of vertices in $G$ reachable from $v$ by a monotone decreasing path (with respect to $\sigma$).
Suppose $(U,\rho)$ is a random rooted locally finite graph (that is, $(U,\rho)$ is a distribution supported on rooted locally finite graphs).
We say that $(U,\rho)$ has \defn{nonexplosive growth} if the past of $\rho$ in $U$, with respect to a random labelling $\sigma$, is almost surely finite.
For such $(U,\rho)$ we may define
\begin{equation*}
  \gir(U,\rho)=\pr[\rho \in \gis_\sigma(U[\past{\rho}])].
\end{equation*}
We say that a graph sequence $G_n$ converges locally to $(U,\rho)$, and denote it by $G_n\toloc(U,\rho)$, if for every $r\ge 0$, the ball of radius $r$ around a uniformly chosen point from $G_n$ converges in distribution to the ball of radius $r$ around $\rho$ in $U$.
To make this notion precise, we need to endow the space of rooted locally finite connected graphs with a topology.
This will be done rigorously in \cref{sec:locallimit}.
The following key tool motivates the definitions above.
\begin{proposition}\label{thm:mean}
  If $G_n\toloc (U,\rho)$ and $(U,\rho)$ has nonexplosive growth then
  $\egir(G_n) \to \gir(U,\rho)$.
\end{proposition}
\noindent
It is easy to see that $\gir(U,\rho)$ is at least $\E[(d(\rho)+1)^{-1}]$;
however, the expected value of $d(\rho)$ may be infinite, even if $(U,\rho)$ is a local limit of a sequence of finite graphs.

~

\Cref{thm:mean}, as a more or less straightforward application of the objective method, can be considered folklore.
However, at present there does not appear to be an explicit statement of the above form in the literature.
For completeness, we provide a short proof in \cref{sec:locallimit}.

~

With some mild growth assumptions on the graph sequence, we also obtain asymptotic concentration of the greedy independence ratio around its mean.
For a graph $G$ let $\cN_G(r)$ be the random variable counting the number of paths of length at most $r$ from a uniformly chosen random vertex of $G$.
For two real numbers $x,y$ denote by $x\wedge y$ their minimum.
For a graph sequence $G_n$, let
\begin{equation*}\label{mu*}
  \mu^*(r) = \lim_{M\to\infty} \limsup_{n\to\infty} \E[\cN_{G_n}(r)\wedge M].
\end{equation*}
We say that $G_n$ has subfactorial path growth (\sfpg{}) if $\mu^*(r) \ll_r r!$ (by $g_1(r)\ll_r g_2(r)$ we mean that $\lim_{r\to\infty}g_1(r)/g_2(r)=0$).
Note that every graph sequence with uniformly bounded degrees has \sfpg{}, but there are graph sequences with unbounded degrees, and even with unbounded average degree, which still have \sfpg{}.
For most cases, and for all of the applications presented in this paper, requiring that the somewhat simpler expression $\limsup_{n\to\infty}\E[\cN_{G_n}(r)]$ is subfactorial would have sufficed; however, requiring that the ``truncated'' mean $\mu^*(r)$ is subfactorial is less strict, and is more natural for the following reason: if the graph sequence converges locally, then $\mu^*(r)$ is the expected number of paths of length at most $r$ in the limit.
In addition, while a sequence of graphs with \sfpg{} does not necessarily have a local limit, it does have a locally convergent subsequence,
and any limit of such a sequence will have nonexplosive growth (see proof of \cref{thm:main}).

~

For two functions $f_1(n),f_2(n)$ write $f_1(n)\sim f_2(n)$ if $f_1(n)=(1+o(1))f_2(n)$.
We are now ready to state our concentration result.

\begin{theorem}\label{thm:main}
  If $G_n$ has \sfpg{} and $G_n\toloc (U,\rho)$ then $\gir(G_n) \sim \gir(U,\rho)$ with high probability.
\end{theorem}
\Cref{thm:main} can (and will) be used as a tool to estimate $\gir(G_n)$ for various graph sequences, as we will see in \cref{sec:applications}.

~

  We remark that
  Gamarnik and Goldberg~\cite{GG10} have already established concentration of   $\gir(G_n)$ around its mean, assuming that the degrees of $G_n$ are uniformly bounded.
  Here we relax that assumption by not even requiring a bounded average degree.
  For a more detailed comparison, see \cref{sec:comp}.

\subsubsection{Locally tree-like graph sequences}

We call a (random) graph sequence \defn{locally tree-like} when the limiting object is supported on rooted trees.
Our next result is a general differential-equations based tool for analysing the asymptotics of the greedy independence ratio of locally tree-like (random) \sfpg{} graph sequences, with the restriction that their limit may be emulated by a \emph{simple} branching process with at most countably many types.
Roughly speaking, a \defn{multitype branching process} is a rooted tree, in which each node is assigned a \emph{type}, and the number and types of each node's ``children'' follow a law that depends solely on the node's type and is independent for distinct nodes.
Such a branching process is called \defn{simple} if each such law is a product measure.
We give formal definitions in \cref{sec:de}.
The following theorem reduces the problem of calculating $\gir(U,\rho)$ in these cases to the problem of solving a (possibly infinite) system of ODEs.
Here, given a (countable) set of types $T$, for every two types $k,j\in T$ we denote by $\mu^{k\to j}$ the distribution of the number of nodes of type $j$ for a parent of type $k$.

\begin{theorem}\label{thm:de}
  Let $(U,\rho)$ be a simple multitype branching process with finite or countable type set $T$,
  root distribution $\dot\mu$
  and offspring distributions $\mu^{k\to j}$.
  For every $x\in[0,1]$ and $k,j\in T$ let $\mu^{k\to j}_x=\bin(\mu^{k\to j},x)$
  denote the distribution
  of the number of children of type $j$ of a node of type $k$ with random label at most $x$.
  Then,
  \begin{equation}\label{eq:iota}
    \gir(U,\rho) = \sum_{k\in T}y_k(1)\dot\mu(k),
  \end{equation}
  where $\{y_k\}_{k\in T}$ is a solution to the following system of ODEs:
  \begin{equation}\label{eq:fundamental}
  y_k'(x) =
  \sum_{\ell \in \NN^T}
  \prod_{j\in T} \mu^{k\to j}_x(\ell_j)
  \left(1-\frac{y_j(x)}{x}\right)^{\ell_j},
  \qquad y_k(0)=0.
  \tag{$*$}
  \end{equation}
\end{theorem}

We call \eqref{eq:fundamental} the \defn{fundamental system of ODEs} of the branching process $(U,\rho$).
While this system of ODEs may seem complicated, in many important cases it reduces to a fairly simple system, as we will demonstrate in \cref{sec:applications}.
In particular, the proof of \cref{thm:de} implies that a solution to \eqref{eq:fundamental} exists, and in the presented applications, it will be unique.
In the cases where $(U,\rho)$ is either a single type branching process or a random tree with \iid{} degrees, we provide an easy probability generating function tool that may be used to ``skip'' solving~\eqref{eq:fundamental}.  This is described in \cref{sec:pgf}.
We mention that a somewhat related, but apparently less applicable statement, providing differential equations for the occupancy probability of a given vertex in bounded degree graphs, appears in~\cite{PS05}.

  Observe that
  the proof of \cref{thm:de} actually yields a stronger result.
  Replacing $y_k(1)$ with $y_k(x)$ in the RHS of \eqref{eq:iota}, the obtained quantity is the probability that the root is occupied ``at time $x$'', namely, when vertices whose label is above $x$ are ignored.

\subsubsection{Applications}
To demonstrate the power and applicability of our method, we compute (in \cref{sec:applications}) the greedy independence ratio for several commonly studied (random) locally tree-like graph sequences.
We do so by first reducing the problem to finding the probability that the root of the local limit of the graph sequence ends in the random greedy independent set (using \cref{thm:main}) and then solving its fundamental system of ODEs, as described in \cref{thm:de}.
In a few cases, where the local limit is either a single-type branching process or a random tree with \iid{} degrees, we are assisted by a probability generating functions based ``trick'' that allows us to ``skip'' solving the differential equations (see \cref{sec:pgf}).

In particular, we calculate the asymptotics of the greedy independence ratio for paths and cycles, recovering classical results of Flory~\cite{Flo39} and Page~\cite{Pag59};
for binomial random graphs, reproving a result of McDiarmid~\cite{McD84} for $p=\Theta(1/n)$;
for uniform spanning trees and random functional digraphs (new results);
for sparse random planar graphs (a new result);
and for random regular graphs (and regular graphs with high girth), recovering results of Wormald~\cite{Wor95} (and Lauer and Wormald~\cite{LW07}).

\paragraph{Hypergraphs}
  As mentioned in the introduction, one may run the random greedy algorithm for producing a maximal independent set on a hypergraph.
  Here, an independent set is a set of vertices that does not span a hyperedge.
  The formal definitions for local convergence (see \cref{sec:locallimit}) and \sfpg{} easily generalise to the hypergraph setting, thus \cref{thm:mean,thm:main} are also valid in this setting.
  In fact, it is not hard to generalise the notion of simple multitype branching processes to represent local limits of locally tree-like hypergraphs.
  Hence (an analogue of) \cref{thm:de} can also be applied in this setting.
  In \cref{sec:hyper}, we discuss how to calculate the asymptotic size of the random greedy maximal independent set on locally tree-like hypergraphs.
  We demonstrate the application of our tools by reproving results from a recent\footnote{The work~\cite{NV21} appeared online after a conference version of this paper~\cite{KMMS20} was posted.} paper by Nie and Verstra\"ete~\cite{NV21}.

\subsubsection{Trees}
We conclude our work by analysing the random greedy MIS in trees.
A plausible guess is that among all trees with a given number of vertices, the path, as an ``opposite'' (in some sense) to the star, would minimise the expected size of the obtained greedy MIS.
Our following theorem makes this intuitive statement formal.

\begin{theorem}\label{thm:trees}
  Let $n\ge 1$, let $T$ be a tree on $n$ vertices and let $P_n$ be the path on $n$ vertices.  Then $\egir(P_n)\le\egir(T)$.
\end{theorem}
This theorem gives us an exact (non-asymptotic) explicit lower bound for the expected greedy independence ratio of trees (an asymptotic upper bound of $1$ is trivial, as can be seen by considering the sequence of stars).
The methods used to prove it are very different from those used in the rest of this paper and are more combinatorial.
In particular, we use a transformation on trees, initially introduced by Csikv\'ari in~\cite{Csi10}, which gives rise to a graded poset of all trees of a given order, in which the path is the unique minimum (say).
While we cannot show that this transformation can only increase the expected greedy independence ratio, we show it can only increase some other quantitative property of trees, which allows us to argue that paths indeed achieve the minimum expected greedy independence ratio.

\subsubsection{Comparison with previous work}\label{sec:comp}
The main goal of this paper is to present a unified, comprehensive, easy-to-apply framework for analysing the performance of local algorithms on (random) (hyper)graphs.
We focus on the simplest sort of such an algorithm: the random greedy MIS.
Our framework has two key components.
The first component (\cref{thm:main}) is essentially an application of the objective method~\cite{AS04}, which concerns any locally convergent graph sequence.
The second component (\cref{thm:de}), applicable only for locally tree-like graph sequences, is machinery for computing $\gir(U,\rho)$ (and thus the limit of $\gir(G_n)$ for many graph classes through \cref{thm:main}) by way of a system of differential equations, which in turn can be solved easily in quite a few cases using probability generating functions (\cref{sec:pgf}).
We wish to emphasise that while the objective method is a known tool that has been applied in the study of several parameters of (random) graph sequences
(see, e.g.,~\cites{Ald92,AS16,vdH21+}),
and although random greedy algorithms for producing an MIS have been thoroughly studied in the past
(see, e.g.,~\cites{Flo39,Pag59,McD84,Wor95,DLM16,BB16,BJL17,BJM17,FN18}),
no explicit and applicable connection has been made between the tool and the process.
Thus, while \cref{thm:mean} may be considered folklore, and while \cref{thm:main} (or, more precisely, \cref{cl:var}, which encapsulates the main content of \cref{thm:main}) was proved, in less generality, by Gamarnik and Goldberg~\cite{GG10}, in the present paper, we relate the concepts of random greedy algorithms, local convergence and branching processes to provide an integrated and applicable framework.
This new framework allows us to easily prove well-established results as well as new ones.

Let us dwell upon the comparison between \cref{thm:main} and the aforementioned result of Gamarnik and Goldberg.
As far as we know, Gamarnik and Goldberg were the first to prove bounds on the variance of the density of the random greedy MIS, hinting that the random greedy algorithm for producing an MIS is very robust.
Their setting assumes that the graph sequence has a uniformly bounded degree, and the bound they obtain on the variance is superexponential in the degree.
While we do not attempt to provide explicit bounds on the variance, we show that it is decaying (tending to zero) regardless of the maximum degree.
In fact, our much weaker assumption of subfactorial path growth allows the graph sequence to have a diverging {\em average} degree.
In addition, while Gamarnik and Goldberg apply the objective method, they do it somewhat implicitly and restrict the application to random/high-girth regular graphs.
We put this on a more formal footing and in the largest possible generality through the notion of local convergence.

Analysis of the random greedy MIS on infinite rooted graphs also appears in the literature, often using different terminology (such as \term{blocking RSA};
see, e.g.,~\cites{FP91,PS05}).
Penrose and Sudbury~\cite{PS05} give forward equations which resemble \cref{thm:de}.
Their equations, unlike ours, are not limited to trees;
on the other hand, they are stated and proved for deterministic bounded degree graphs, and, in any case, they appear to be impractical for graphs with cycles.
They later apply the forward equations for $d$-regular trees (\term{Bethe lattices}), reproving known results (see, e.g.,~\cite{FP91}).
We obtain these results as a special case of \cref{thm:de}, and project them to random/high-girth regular graphs.
For analysis of the random greedy MIS on random trees, we refer the reader to the works of Dehling, Fleurke and Kulske~\cite{DFK08} (a result we reprove and slightly generalise in \cref{sec:pgf}), and of Sudbury~\cite{Sud09}.

Our final result, \cref{thm:trees}, concerns an exact (non-asymptotic) analysis of the density of the random greedy MIS on trees.
There are many known nontrivial graph parameters that the path minimises among all trees on the same number of vertices.
For example, Jamison~\cite{Jam81} showed that the expected size of a random subtree of a tree attains its minimum on the path.
Csikv\'ari~\cite{Csi10} and later Bollob\'as and Tyomkyn~\cite{BT12} proved that the path minimises the number of walks of a given length (and thus also its \term{spectral radius}).
For their results, they study a certain transformation on the set of all trees of a given size, called the KC-transformation (see \cref{sec:KC}), and show that it gives rise to a graded poset in which the path is the unique minimum, and the star is the unique maximum.
Their results are obtained then by showing that the parameter in question is monotone with respect to that poset.
In our work, we exploit the same transformation.

\subsection{Organisation of the paper}\label{sec:organisation}
We start with formal definitions and proofs of the main results.
We introduce the metric that is used to define the notion of \emph{local convergence} in \cref{sec:locallimit}, where we also prove \cref{thm:mean}.
In \cref{sec:conc}, we prove \cref{thm:main} by essentially proving a decay of correlation between vertices in terms of their distance and showing that typical pairs of vertices are distant.
In fact, the results of \cref{sec:conc} imply that even without local convergence, under mild growth assumptions, the variance of the greedy independence ratio is decaying.

In \cref{sec:de}, we focus our attention on locally tree-like graph sequences, define (simple, multitype) branching processes, and prove \cref{thm:de}.
We enhance this in \cref{sec:pgf} by introducing a probability generating functions based ``trick'', which allows, in some cases, a significant simplification.

We continue by presenting an extensive list of important applications in~\cref{sec:applications}, where we prove some new results and reprove some known ones, using the machinery of \cref{thm:main,thm:de}.  In a few cases, we are assisted by the claims from \cref{sec:pgf}.
In \cref{sec:hyper}, we demonstrate how the presented tools work, almost as-is, for locally tree-like {\em hypergraphs} (reproving results from~\cite{NV21}). 

In \cref{sec:trees} we focus further on trees, where we prove \cref{thm:trees}.
To this end, we pinpoint several interesting properties of the expected greedy independence ratio of the path.

\section{Local convergence}\label{sec:locallimit}
In order to study asymptotics, it is often useful to construct a suitable
limiting object first.
Local limits were introduced by Benjamini and Schramm~\cite{BS01}
and studied further by Aldous and Steele~\cite{AS04} (A very similar approach has already been introduced by Aldous in~\cite{Ald91}).
Local limits, when they exist, encapsulate the asymptotic data of local behaviour of the convergent graph sequence, and in particular, that of the performance of the greedy algorithm.

We start with basic definitions, which we define for graphs but that can be extended to hypergraphs in an obvious way.
Consider the space $\Gdot$ of rooted locally finite connected graphs viewed up
to root preserving graph isomorphisms.
We provide $\Gdot$ with the metric $\dloc((G_1,\rho_1),(G_2,\rho_2)) = 2^{-R}$,
where $R$ is the largest integer for which
$B_{G_1}(\rho_1,R)\simeq B_{G_2}(\rho_2,R)$.
Here we understand $B_{G}(\rho,R)$ as the \emph{rooted} subgraph of $(G,\rho)$
spanned by the vertices of distance at most $R$ from $\rho$, and $\simeq$ as \emph{rooted-isomorphic}.
It is an easy fact that $(\Gdot,\dloc)$ is a separable complete metric space,
hence it is a Polish space (see~\cite{BS01}).
$(\Gdot,\dloc)$, while being bounded, is not
compact (the sequence of rooted stars $S_n$ does not have a convergent
subsequence).

Recall that a sequence of random elements $\{X_n\}_{n=1}^\infty$
\textbf{converges in distribution} to a random element $X$, if for every
bounded continuous function $f$ we have that $\E[f(X_n)]\to\E[f(X)]$.
Let $G_n$ be a sequence of (random) finite graphs.  We say that $G_n$
\defn{converges locally} to a (random) element $(U,\rho)$ of $\Gdot$
if for every $r\ge 0$,
the sequence $B_{G_n}(\rho_n,r)$ converges in distribution to $B_U(\rho,r)$,
where $\rho_n$ is a uniformly chosen vertex of $G_n$.
Since the inherited topology on all rooted balls in $\Gdot$ with radius $r$
is discrete, this implies convergence in total variation distance.

~

We are now ready to prove \cref{thm:mean}.
\begin{proof}[Proof of \cref{thm:mean}]
  Fix $\eps>0$.
  Let $\sigma$ be a random labelling of (a random sample of) $U$,
  and let $\ell_\sigma$ be the length of the longest decreasing sequence (w.r.t.\ $\sigma$)
  starting from $\rho$.
  Since $(U,\rho)$ has nonexplosive growth, there exists $r_\eps$ for which
  $\pr[\ell_\sigma\ge r_\eps]<\eps$.
  For $n\ge 1$, let $\rho_n$ be a uniformly chosen random vertex of $G_n$,
  and let $\pi$ be a uniform random permutation of its vertices.
  For $r\ge 0$, denote $G_n^r=B_{G_n}(\rho_n,r)$ and $U^r=B_U(\rho,r)$.
  We couple $(G_n,\rho_n,\pi)$
  with $(U^r,\sigma)$ 
  as follows.
  Since $G_n^r$ converges in distribution
  (and hence in total variation distance)
  to $U^r$, there exists $n_r$ such that for all $n\ge n_r$
  we have a coupling between $(G_n,\rho_n)$ and $U^r$ for which $\pr[G_n^r\not\simeq U^r]\le \eps$.
  Assuming $G_n^r\simeq U^r$,
  let $\varphi:G_n^r\to U^r$ be an isomorphism,
  let $\pi^r$ be the permutation on the vertices of $G_n^r$
  which agrees with the ordering of the labels on the vertices of the isomorphic image
  (that is, $\pi^r_u<\pi^r_v \iff \sigma_{\varphi(u)}<\sigma_{\varphi(v)}$).
  Observe that $\pi^r$ has a uniform law.
  Now, since $\pi$ induces a uniform random permutation of $G_n^r$ (by restriction),
  we may couple $\pi$ with $\pi^r$ such that $\pi^r$ is a restriction of $\pi$.
  Note that under this coupling, if it succeeds,
  $\rho_n\in\gis(G_n^r)\iff \rho\in\gis_\sigma(U^r)$
  (here we understand $\gis(G_n^r))$ as induced by $\pi^r$).
  However, on the event
  ``$\ell_\sigma \le r$'', $\rho_n\in\gis(G_n^r)\iff \rho_n\in\gis(G_n)$
  (here we understand $\gis(G_n))$ as induced by $\pi$)
  and
  $\rho\in\gis_\sigma(U^r)\iff \rho\in\gis_\sigma(U[\past{\rho}])$.  Observing that
  $\egir(G_n)=\pr[\rho_n\in\gis(G_n)]$ we obtain that for $r\ge r_\eps$ and $n\ge n_r$,
  $|\egir(G_n)-\gir(U,\rho)|<2\eps$.
\end{proof}

\section{Concentration via exploration--decision algorithms}\label{sec:conc}
With some mild growth assumptions on the graph sequence, without assuming local convergence, we obtain asymptotic concentration of the greedy independence ratio around its mean.
Under these assumptions we show that the dependence between the inclusion of distinct nodes in the maximal independent set decays as a function of their distance, a phenomenon which is sometimes called \emph{correlation decay} or \emph{long-range independence}.
To prove that the model exhibits this phenomenon, we show that with high probability there are no ``long'' monotone paths emerging from a typical vertex, which is the content of the next claim.
We then observe that two independent random vertices are typically distant, and use a general lemma about exploration algorithms to prove decay of correlation.
We remark that similar locality arguments appear in~\cite{NO08}.

\begin{claim}\label{cl:past:rad}
  Suppose that $G_n$ has \sfpg{}.
  Let $\pi$ be a uniform random permutation of the vertices of $G_n$, and let $u$ be a uniformly chosen vertex from $G_n$.
  Then, for every $\eps>0$, there exists $r>0$ such that for every large enough $n$, the probability that there exists a monotone decreasing path of length $r$ (w.r.t.\ $\pi$), emerging from $u$, is at most $\eps$.
\end{claim}

\begin{proof}
  Let $\eps\ge 0$.
  Since $\mu^*(r)\ll_r r!$ for every large enough $r$ we have $\mu^*(r)\le \eps r!$.
  We couple $\cN_{G_n}(r)$ and $u$ such that the former counts the number of paths of length at most $r$ emerging from the latter.
  Denote by $A_n^r$ the event that there exists a monotone decreasing path in $G_n$ (w.r.t.\ $\pi$) emerging from $u$ of length $r$.
  Evidently, the probability that a given path of length $r$ is monotone decreasing w.r.t.\ $\pi$ is $1/r!$.
  Since $\mu^*(r)$ is finite, there exists $M\ge 0$ such that $\pr[\cN_{G_n}(r)>M]<\eps$ for every large enough $n$.
  In addition, for large enough $n$ we have
  $\E[\cN_{G_n}(r)\wedge M] \le 2\mu^*(r)$.
  Hence, for large enough $n$,
  \begin{align*}
    \pr[A_n^r]
    &\le \sum_{m=0}^M \pr[A_n^r \mid \cN_{G_n}(r) = m]
           \cdot \pr[\cN_{G_n}(r) = m]
       + \pr[\cN_{G_n}(r) > M]\\
    &\le \frac{1}{r!} \cdot \E[\cN_{G_n}(r)\wedge M] + \eps \le 3\eps.
    \qedhere
  \end{align*}
\end{proof}

\begin{claim}\label{cl:dist}
  Suppose that $G_n$ has \sfpg{}.
  Let $u,v$ be two independently and uniformly chosen vertices from $G_n$.
  Then, for every $\eps,r\ge 0$ we have that for every large enough $n$,
    $\pr[\dist_{G_n}(u,v)\le r] \le \eps$.
\end{claim}

\begin{proof}
  Let $\eps,r\ge 0$.
  We couple $\cN_{G_n}(r)$ and $u$ such that the former counts the number of paths of length at most $r$ emerging from the latter.
  Note that under this coupling, $|B_{G_n}(u,r)|\le\cN_{G_n}(r)$.
  Since $\mu^*(r)$ is finite, there exists $M\ge 0$ such that $\pr[\cN_{G_n}(r)>M]<\eps$ for every large enough $n$.
  Hence, for large enough $n$,
  \begin{align*}
    &\pr[\dist_{G_n}(u,v)\le r]
    = \pr[v\in B_{G_n}(u,r)]\\
    &\le \pr[v\in B_{G_n}(u,r) \mid \cN_{G_n}(r) \le M]
       + \pr[\cN_{G_n}(r) > M]
     \le \frac{M}{n}+\eps \le 2\eps.\qedhere
  \end{align*}
\end{proof}

Let $G=(V,E)$ be a graph.  An \defn{exploration--decision rule} for $G$ is a (deterministic) function $\cQ$, whose input is a pair $(S,g)$, where $S$ is a non-empty sequence of distinct vertices of $V$, and $g:S\to[0,1]$, and whose output is either a vertex $v\in V\sm S$ or a ``decision'' $\T$ or $\F$.
An \defn{exploration--decision algorithm} for $G$, with rule $\cQ$, is a (deterministic) algorithm $\sA$, whose input is an initial vertex $v\in V$ and a function $f:V\to[0,1]$, which outputs $\T$ or $\F$, and operates as follows.
Set $u_1=v$.
Suppose $\sA$ has already set $u_1,\ldots,u_i$.
Let $x=\cQ((u_1,\ldots,u_i),f\restriction_{\{u_1,\ldots,u_i\}})$.
If $x\in V$, set $u_{i+1}=x$ and continue.
Otherwise stop and return $x$.
We call the set $u_1,\ldots,u_i$ at this stage the \defn{range} of the algorithm's run.
We denote the output of the algorithm by $\sA(v,f)$ and its range by
$\rng_{\sA}(v,f)$.
The radius of the algorithm's run, denoted $\rad_{\sA}(v,f)$, is the maximum distance between $v$ and an element of its range.

\begin{lemma}\label{lem:corr}
  Let $\eps>0$.
  Let $G=(V,E)$ be a graph, let $\sigma$ be a random labelling of its vertices, let $\sA$ be an exploration--decision algorithm for $G$ and let $r\ge 1$.
  Let $u,v$ be sampled independently (and independently of $\sigma$)
  from some distribution over $V$.
  Suppose that w.p.\ at least $1-\eps$ both $\dist_G(u,v)\ge 3r$, and $\rad_{\sA}(u,\sigma),\rad_{\sA}(v,\sigma)\le r$.
  Then $|\cov[\sA(u,\sigma),\sA(v,\sigma)]|=O(\eps)$.
\end{lemma}

\begin{proof}
  Let $\cQ$ be the rule of the algorithm $\sA$.
  The \defn{$r$-truncated} version of $\cQ$, denoted $\cQ^r$, is defined as follows.
  To determine $\cQ^r((u_1,\ldots,u_i),g)$, $\cQ^r$ checks the value $x=\cQ((u_1,\ldots,u_i),g)$.
  If $x\in\{\T,\F\}$ or $\dist_G(u_1,x)\le r$, $\cQ$ returns $x$.
  Otherwise, it returns $\F$.
  The \defn{$r$-truncated} version of the algorithm $\sA$, denoted $\sA^r$,
  is the exploration--decision algorithm with rule $\cQ^r$.
  Note that for every $v$ and $f$, $\rad_{\sA^r}(v,f)\le r$.

  For a vertex $w\in \{u,v\}$, let $X_w$ be the event ``$\sA(w,\sigma)=\T$'', let $Y_w$ be the event ``$\sA^r(w,\sigma)=\T$'', and let $r_w=\rad_{\sA}(w,\sigma)$.
  Note that $\pr[X_w\land r_w\le r]=\pr[Y_w\land r_w\le r]=\pr[Y_w]$, thus
  $\pr[X_w]=\pr[Y_w]+O(\eps)$.  Since for $x,y$ satisfying
  $\dist_G(x,y)\ge 3r$ we have that $Y_x,Y_y$ are independent, it follows
  that $\pr[Y_u\land Y_v]=\pr[Y_u]\pr[Y_v]+O(\eps)$.
  \begin{align*}
  \pr[X_u\land X_v]
  &= \pr[X_u\land X_v\land (\max\{r_u,r_v\}\le r)]
  +\pr[X_u\land X_v\land (\max\{r_u,r_v\}> r)]\\
  &= \pr[Y_u\land Y_v\land (\max\{r_u,r_v\}\le r)] + O(\eps)\\
  &= \pr[Y_u\land Y_v] + O(\eps)
  = \pr[Y_u]\pr[Y_v] + O(\eps) = \pr[X_u]\pr[X_v]+O(\eps).
  \qedhere
  \end{align*}
\end{proof}

We now apply the lemma in our setting.

\begin{claim}\label{cl:cov}
  Suppose that $G_n$ has \sfpg{}.
  Let $u,v$ be two independently and uniformly chosen vertices from $G_n$.
  Denote by $R_u,R_v$ the events that $u\in\gis(G_n)$, $v\in\gis(G_n)$, respectively.  Then $|\cov[R_u,R_v]|=o(1)$.
\end{claim}

\begin{proof}
  Let $\eps>0$.
  We describe an exploration--decision algorithm $\sA$ by defining its rule.
  Given a vertex sequence $S=(u_1,\ldots,u_i)$ and labels $g:S\to[0,1]$, the rule checks for monotone decreasing sequences emerging from $u_1$, in $S$, with respect to $g$.  Denote by $\cE$ the set of ends of these sequences.  If there are vertices in $V\sm S$ with neighbours in $\mathcal{E}$, return an arbitrary vertex among these.  Otherwise, perform the Greedy MIS algorithm on the past of $u_1$ inside $S$, and return $\T$ if $u_1$ ends up in the MIS, or $\F$ otherwise.
  We observe that if $\sigma$ is a random labelling of $G_n$ then for $w\in\{u,v\}$ the event $\sA(w,\sigma)=\T$ is in fact the event $R_w$.
  We also note that if the longest monotone decreasing sequence, w.r.t.\ $\sigma$, emerging from $w$ is of length $r-1$, then $\rad_{\sA}(w,\sigma) \le r$.

  By \cref{cl:past:rad} there exists $r>0$ such that for every large enough $n$ the probability that there exists a monotone decreasing path of length $r-1$ from either $u$ or $v$ is at most $\eps$.
  By \cref{cl:dist}, for large enough $n$, the probability that the distance between $u$ and $v$ is at most $3r$ is at most $\eps$.
  Therefore, by \cref{lem:corr}, $|\cov[\sA(u,\sigma),\sA(v,\sigma)]|=o(1)$.
\end{proof}

\begin{claim}\label{cl:var}
  Suppose that $G_n$ has \sfpg{}.
  Then $\Var[\gir(G_n)]=o(1)$.
\end{claim}

\begin{proof}
  For a vertex $w$, denote by $R_w$ the event that $w\in\gis(G_n)$.
  Let $u,v$ be two independently and uniformly chosen vertices from $G_n$.
  Since the random variables $\E[R_u\mid u]$ and $\E[R_v\mid v]$ are independent
  (since they are measurable with respect to $u$ and $v$, respectively, which are independent),
  by \cref{cl:cov}, and by the law of total covariance,
  \begin{align*}
    \Var[\gir(G_n)]
    &= \frac{1}{n^2}\sum_{x,y\in V(G_n)} \cov[R_x,R_y]
     = \E[\cov[R_u,R_v\mid u,v]]\\
    &= \cov[R_u,R_v] - \cov[\E[R_u\mid u,v],\E[R_v\mid u,v]]\\
    &= \cov[R_u,R_v] - \cov[\E[R_u\mid u],\E[R_v\mid v]]
    = \cov[R_u,R_v] = o(1).\qedhere
  \end{align*}
\end{proof}

\begin{proof}[Proof of \cref{thm:main}]
  Let $\eps>0$.
  First, we note that since $G_n$ has \sfpg{}, $(U,\rho)$ has nonexplosive growth.
  Indeed, the number of paths of length $r$ in $U$ emerging from $\rho$ is subfactorial in $r$, hence the probability of having a monotone path of length $r$ emerging from the root decays to $0$ as $r$ grows.
  Thus, by \cref{thm:mean}, there exists $n_0$ such that for every $n\ge n_0$, $|\egir(G_n)-\gir(U,\rho)|\le\eps$.
  Thus, by Chebyshev's inequality and \cref{cl:var},
  \begin{equation*}
    \pr[|\gir(G_n)-\gir(U,\rho)| > 2\eps]
    \le \pr[|\gir(G_n)-\egir(G_n)| > \eps]
    \le \eps^{-2}\Var[\gir(G_n)] = o(1).\qedhere
  \end{equation*}
\end{proof}

\section{Branching processes and differential equations}\label{sec:de}
As promised, we give a formal definition of multitype branching processes.
Let $T$ be a finite or countable set, which we call the \defn{type set}.
Let $\dot\mu$ be a distribution on $T$, which we call the \defn{root distribution}, and for each $k\in T$, let $(\mu^{k\to j})_{j\in T}$ be an \defn{offspring distribution}, which is a distribution on vectors with nonnegative integer coordinates.
Let $\tau\sim\dot\mu$, and for every finite sequence of natural numbers $\vect{v}$ let $(\xi^{k\to j}_{\vect{v}})_{j\in T}\sim(\mu^{k\to j})_{j\in T}$ be a random vector, where these random vectors are independent for different indices $\vect{v}$ and are independent of $\tau$.
A \defn{multitype branching process} $(\vect{Z}_t)_{t\in\NN}$ with type set $T$, root distribution $\dot\mu$ and offspring distributions $(\mu^{k\to j})_{j\in T}$ is a Markov process on labelled trees, in which each vertex is assigned a type in $T$, which may be described as follows.
At time $t=0$, the tree $\vect{Z}_0$ consists of a single vertex of type $\tau$, labelled by the empty sequence.
At time $t+1$, the tree $\vect{Z}_{t+1}$ is obtained from $\vect{Z}_t$ as follows.
For each $k\in T$ and $\vect{v}$ of length $t$ and type $k$ in $\vect{Z}_t$, we add the vertices $\vect{v}\concat i$ for all $0\le i < \smash{\sum_{j\in T}\xi^{k\to j}_{\vect{v}}}$, having exactly $\xi^{k\to j}_{\vect{v}}$ of them being assigned type $j$, uniformly at random, and connecting them with edges to $\vect{v}$.\footnote{By $\vect{v}\concat i$ we mean the sequence obtained from $\vect{v}$ by appending the element $i$.}
If in addition $(\mu^{k\to j})_{j\in T}$ is a product measure, namely, if $\xi^{k\to j}_{\vect{v}}\sim\mu^{k\to j}$ are sampled independently for distinct $j\in T$, the process is called \defn{simple}.
We often think of a multitype branching process as the possibly infinite (random) rooted graph $\vect{Z}_\infty=\bigcup_{t\ge 0}\vect{Z}_t$, rooted at the single vertex of $\vect{Z}_0$.

\begin{proof}[Proof of \cref{thm:de}]
  Let $\sigma$ be a random labelling of $U$.
  To ease notation, set $\gir=\gir(U,\rho)$ and $\gis=\gis(U[\past{\rho}])$, and recall that $\gir=\pr[\rho\in\gis]$.
  Let $\tau\sim\dot\mu$ be the type of the root.
  For $k\in T$ and $x\in[0,1]$, define
  $\gir^{(k)} = \pr[\rho\in \gis\mid \tau=k]$ and $\gir^{(k)}_x = \pr[\rho\in \gis\mid \sigma_\rho=x,\ \tau=k]$.
  Note that this is well defined, even if the event that $\sigma_\rho=x$ has probability $0$.
  Let further
  \begin{equation*}
    \gir^{(k)}_{<x} = \int_0^x \gir^{(k)}_zdz,
  \end{equation*}
  so $\gir^{(k)} = \gir^{(k)}_{<1}$, hence
  \begin{equation*}
    \gir = \sum_{k\in T} \gir_{<1}^{(k)} \cdot \pr[\tau=k].
  \end{equation*}
  It, therefore, suffices to show that the family $y_k(x):=\gir^{(k)}_{<x}$ satisfies \eqref{eq:fundamental} (obviously, it satisfies the boundary conditions).
  The critical observation is that conditioning on the label of the root, distinct children in its past are roots to independent randomly labelled subtrees.
  In particular, conditioning on $\sigma_\rho$ and on the event that $v_1,\ldots,v_a$ are the children of $\rho$ in its past, the events ``$v_i\in\gis$'' for $i=1,\ldots,a$ are mutually independent.
  Since $\rho\in\gis$ if and only if $v_i\notin\gis$ for every $i=1,\ldots,a$,
  \begin{align*}
    y_k'(x)
    = (\gir^{(k)}_{<x})'
    = \gir^{(k)}_x
    &= \sum_{\ell \in \NN^T}
       \prod_{j\in T} \mu^{k\to j}_x(\ell_j)
       \left(1-\pr[\rho\in\gis\mid \sigma_\rho<x,\ \tau=j]\right)^{\ell_j}\\
    &= \sum_{\ell \in \NN^T}
       \prod_{j\in T} \mu^{k\to j}_x(\ell_j)
       \left(1-\frac{y_j(x)}{x}\right)^{\ell_j}.\qedhere
  \end{align*}
\end{proof}

\section{Probability generating functions}\label{sec:pgf}
In this section we demonstrate how generating functions may aid solving the fundamental system of ODEs~\eqref{eq:fundamental} (and thus finding $\gir$) for certain simple branching processes.
In the following sections, we will use the notation $y_k(x)$ as in \eqref{eq:fundamental}, and omit the subscript $k$ when the branching process has a single type.

\paragraph*{Single Type Branching Processes}
For a probability distribution $\vect{p}=(p_d)_{d=0}^\infty$, let $\ary{\vect{p}}$ be the $\vect{p}$-ary tree, namely, it is a (single type) branching process, for which the offspring distribution is $\vect{p}$.
The fundamental ODE in this case is
\begin{equation}\label{eq:p:ary}
  y'(x) =
  \sum_{d=0}^{\infty}p_d
  \sum_{\ell=0}^d
  \binom{d}{\ell}
  (1-x)^{d-\ell}x^\ell\left(1-\frac{y(x)}{x}\right)^\ell
  = \sum_{d=0}^{\infty}p_d \left(1-y(x)\right)^d.
\end{equation}
This differential equation may not be solvable, but in many important cases it is, and we will use it.
Denote by $g_\vect{p}(z)$ the probability generating function (\pgf{}) of
$\vect{p}$, that is,
\begin{equation}\label{eq:pgf}
  g_\vect{p}(z) = \sum_{d=0}^\infty p_d z^d.
\end{equation}
Let $\mathsf{h}_\vect{p}(x)$ be the solution to the equation
\begin{equation}\label{eq:h:ary}
  \int_{\mathsf{h}_{\vect{p}}(x)}^1 \frac{dz}{g_\vect{p}(z)} = x.
\end{equation}

\begin{claim}\label{cl:pgf:ary}
  $y(x) = 1 - \mathsf{h}_\vect{p}(x)$.
\end{claim}
\begin{proof}
  Fix $x\in[0,1]$, let $\mathsf{h}=\mathsf{h}_\vect{p}(x)$ and
  $g(z)=g_\vect{p}(z)$.
  Define $\varphi:[0,\beta]\to[\mathsf{h},1]$, where
  $\beta=y^{-1}(1-\mathsf{h})$, as follows:
  $\varphi(u)=1-y(u)$.  Note that by \eqref{eq:p:ary},
  \begin{equation*}
    \varphi'(u) = -y'(u) = -g(\varphi(u)).
  \end{equation*}
  Thus
  \begin{equation*}
    x = \int_{\mathsf{h}}^1 \frac{dz}{g(z)}
    = -\int_{\varphi(0)}^{\varphi(\beta)} \frac{dz}{g(z)}
    = -\int_0^\beta \frac{\varphi'(z) dz}{g(\varphi(z))} = \beta,
  \end{equation*}
  hence $y(x) = 1-\mathsf{h}$.
\end{proof}

In particular, it follows from \cref{cl:pgf:ary} that
$\gir(\ary{\vect{p}}) = 1 - \mathsf{h}_\vect{p}(1)$.

\paragraph*{Random Trees With IID Degrees}
For a probability distribution $\vect{p}=(p_d)_{d=1}^\infty$, let
$\reg{\vect{p}}$ be the $\vect{p}$-tree, namely, it is a random tree in which
the degrees of the vertices are independent random variables with distribution
$p$.  We may view it as a two-type branching process, with type $0$ for the
root and $1$ for the rest of the vertices.  Let $g_\vect{p}(z)$ be the
\pgf{} of $\vect{p}$ (see~\eqref{eq:pgf}, and note that $p_0=0$).
The fundamental system of ODEs in this case is
\begin{equation}\label{eq:p:iid}
  y_0'(x) =
\sum_{d=1}^\infty p_d
\sum_{\ell=0}^d
\binom{d}{\ell}
(1-x)^{d-\ell}x^\ell\left(1-\frac{y_1(x)}{x}\right)^\ell
= \sum_{d=1}^\infty p_d \left(1-y_1(x)\right)^d
= g_\vect{p}(1-y_1(x)),
\end{equation}
and by \eqref{eq:p:ary},
\begin{equation}\label{eq:p:iid:ary}
  y_1'(x) = \sum_{d=0}^\infty p_{d+1}(1-y_1(x))^d
  = \frac{1}{1-y_1(x)}\sum_{d=1}^\infty p_d(1-y_1(x))^d
  = \frac{g_\vect{p}(1-y_1(x))}{1-y_1(x)}.
\end{equation}
Let $\mathfrak{h}_\vect{p}(x)$ be the solution to the equation
\begin{equation*}
\int_{\mathfrak{h}_{\vect{p}}(x)}^1 \frac{zdz}{g_\vect{p}(z)} = x.
\end{equation*}

The next claim is \cite{DFK08}*{Theorem 1}.\footnote{In \cite{DFK08} the authors required that the the degrees of the tree are all at least $2$; we do not require this here.}
\begin{claim}\label{cl:pgf:iid}
  $y_0(x) = \frac{1}{2}\left(1 - \mathfrak{h}^2_\vect{p}(x)\right)$.
\end{claim}
\begin{proof}
  Fix $x\in[0,1]$, let $\mathfrak{h}=\mathfrak{h}_\vect{p}(x)$ and
  $g(z)=g_\vect{p}(z)$.
  Define $\varphi:[0,\beta]\to[\mathfrak{h},1]$, where
  $\beta=y_1^{-1}(1-\mathfrak{h})$, as follows:
  $\varphi(u)=1-y_1(u)$.  Note that by \eqref{eq:p:iid:ary},
  \begin{equation*}
  \varphi'(u) = -y_1'(u) = -\frac{g(\varphi(u))}{\varphi(u)}.
  \end{equation*}
  Thus
  \begin{equation*}
  x = \int_{\mathfrak{h}}^1 \frac{zdz}{g(z)}
  = -\int_{\varphi(0)}^{\varphi(\beta)} \frac{zdz}{g(z)}
  = -\int_0^\beta \frac{\varphi'(z) \varphi(z)dz}{g(\varphi(z))} = \beta,
  \end{equation*}
  hence $y_1(x) = 1-\mathfrak{h}$.  From~\eqref{eq:p:iid}
  and~\eqref{eq:p:iid:ary} it follows that
  $y_0'(x)=g(\mathfrak{h}) =
  y_1'(x)\cdot\mathfrak{h}=-\mathfrak{h}\mathfrak{h}'$, and since $y_0(0)=0$ it
  follows that $y_0(x) = \frac{1}{2}\left(1-\mathfrak{h}^2\right)$.
\end{proof}

In particular, it follows from \cref{cl:pgf:iid} that
$\gir(\reg{\vect{p}}) = \frac{1}{2}\left(1 - \mathfrak{h}^2_\vect{p}(1)\right)$.

\section{Applications}\label{sec:applications}
The goal of this section is to demonstrate the power of the introduced framework by calculating the greedy independence ratio for several natural (random) graph sequences.
We do so by finding their local limit and solving its fundamental system of ODEs, as described in \cref{thm:de}.
In some cases, we may use probability generating functions to ease calculations as described in \cref{sec:pgf}.
In the following, we analyse the process in the setting of various commonly studied (random) graph sequences.
To highlight the method's applicability, we focus on cases where computational difficulties are minimal.

Note that not all measures on rooted graphs arise as local limits of finite graphs.
One necessary condition for a measure to be such a limit
is captured, informally, by the property that the root
is ``equally likely to be any vertex'', even though the graph may be infinite.
This notion can be made rigorous;
see, e.g., \cite{AL07} or \cite{Bor16}, {Chapter 3}.
Measures that have this property are called \defn{unimodular}\footnote{%
The question of whether every unimodular measure is a local limit of finite graphs is open.}.
In our analysis, we also consider non-unimodular random rooted graphs.
The reason for this is twofold:
(a) we wish to demonstrate applications of \cref{thm:de} and the methods described in \cref{sec:pgf}
    in various settings; and
(b) we apply results in non-unimodular settings in our analysis of unimodular measures.

As a final remark, we wish to stress that our list of applications is not intended to be exhaustive.
In particular, we do not consider random graphs with a given degree sequence,
a case that was analysed in \cites{BJM17,BJL17}.

\subsection{Infinite-Ray Stars}\label{sec:istar}
For $d\ge 1$, let $\istar{d}$ be the \defn{infinite-ray star}
with $d$ branches.
Formally, the vertex set of $\istar{d}$ is
$\{(0,0)\}\cup\{(i,j):i\in[d],j=1,2,\ldots\}$, and $(i,j)\sim(i',j')$ if
$|j-j'|=1$ and either $i=i'$ or $ii'=0$.  Note that $\istar{1}=\NN$ and
$\istar{2}=\ZZ$.  This is a two-type branching process, with types $d$ for the
root and $1$ for a branch vertex.
The fundamental system of ODEs in this case is $y_d'(x)=(1-y_1(x))^d$,
and for $d=1$ we obtain the equation $y_1'=1-y_1$ of which the solution is
$y_1(x)=1-e^{-x}$.
For $d>1$ we obtain the equation $y_d'=e^{-dx}$ of which the solution is
$y_d(x)=\frac{1}{d}(1-e^{-dx})$.
Since $\tau=d$ a.s., it follows that
$\gir(\istar{d})=y_d(1)=\zeta_d:=\frac{1}{d}(1-e^{-d})$.
In particular,
$\gir(\NN)=1-e^{-1}\approx 0.6321...$ and
$\gir(\ZZ)=\frac{1}{2}(1-e^{-2})\approx 0.43233...$.

As $\NN$ is a single type branching process and $\ZZ$ is a random tree with \iid{} degrees, we may use the alternative approach for calculating $\gir(\NN)$ and $\gir(\ZZ)$, as described in \cref{sec:pgf}.
Solving $\int_h^1 \frac{dz}{z}=1$ gives $h=e^{-1}$, hence by \cref{cl:pgf:ary}, $\gir(\NN)=1-e^{-1}$, and by \cref{cl:pgf:iid}, $\gir(\ZZ)=\frac{1}{2}\left(1-e^{-2}\right)$.

\paragraph{Paths and cycles}
The local limit of the sequences $P_n$ of paths and $C_n$ of cycles is $\ZZ$ (rooted arbitrarily).
It follows from the discussion above that $\gir(P_n),\gir(C_n)\sim\frac{1}{2}(1-e^{-2})$ \whp{}.
This asymptotic density was already calculated by Flory~\cite{Flo39} (who only considered the expected ratio) and independently by Page~\cite{Pag59} and can be thought of as the discrete variant of R\'enyi's parking constant (see~\cite{MathC}).
We remark that a somewhat similar asymptotic analysis of the random greedy maximal independent set on the path using an analogous process on $\ZZ$ appears in~\cite{Ger15}.

\subsection{Poisson Galton--Watson Trees}\label{sec:gw:pois}
A Poisson Galton--Watson tree $\gw{\lambda}$ is a single type branching process with offspring distribution $\pois(\lambda)$ for some parameter $\lambda\in(0,\infty)$.
The fundamental ODE in this case is $y'(x)=e^{-\lambda y(x)}$
(This can be calculated directly using \eqref{eq:p:ary}).
The solution for this differential equation is $y(x)=\ln(1+\lambda x)/\lambda$,
hence $\gir(\gw{\lambda})=y(1)=\ln(1+\lambda)/\lambda$.
The same result can be obtained using the probability generating function of
the Poisson distribution, as described in \cref{sec:pgf}.

\paragraph{Binomial random graphs}
Consider the binomial random graph $G(n,\lambda/n)$, which is the graph on $n$ vertices in which every pair of nodes is connected by an edge independently with probability $\lambda/n$.
It is easy to check that it converges locally to $\gw{\lambda}$
(see, e.g.,~\cite{CurRG}), hence $\gir(G(n,\lambda/n))\sim\ln(1+\lambda)/\lambda$ \whp{}, recovering a known result
(see~\cite{McD84}).

\subsection{Size-Biased Poisson Galton--Watson Trees}\label{sec:sbgw}
For $0<\lambda\le 1$, a size-biased Poisson Galton--Watson tree
$\sbgw{\lambda}$ can be
defined (see \cite{LPP95}) as a two-type simple branching process, with types
$\mathsf{s}$ (\emph{spine} vertices) and $\mathsf{t}$ (\emph{tree} vertices),
where a spine vertex has $1$ spine child plus $\pois(\lambda)$ tree children,
a tree vertex has $\pois(\lambda)$ tree children, and the root is a spine vertex (when $\lambda=1$, this is sometimes called the \emph{skeleton tree}).
The fundamental system of ODEs in this case is
\begin{align*}
y_\mathsf{s}'(x) &=
x\sum_{d=0}^\infty \frac{(\lambda x)^d}{e^{\lambda x} d!}
\left(1-\frac{y_\mathsf{s}(x)}{x}\right)
\left(1-\frac{y_\mathsf{t}(x)}{x}\right)^d
+ (1-x)\sum_{d=0}^{\infty}\frac{(\lambda x)^d}{e^{\lambda x} d!}
\left(1-\frac{y_\mathsf{t}(x)}{x}\right)^d\\
&= \left(1-y_\mathsf{s}(x)\right)
\sum_{d=0}^\infty \frac{(\lambda x)^d}{e^{\lambda x} d!}
\left(1-\frac{y_\mathsf{t}(x)}{x}\right)^d
= \left(1-y_\mathsf{s}(x)\right) e^{-\lambda y_\mathsf{t}(x)},
\end{align*}
and from \cref{sec:gw:pois} we obtain $y_{\mathsf{t}}(x)=\ln(1+\lambda x)/\lambda$.
Hence
$y_{\mathsf{s}}'(x) = (1-y_{\mathsf{s}}(x))/(1+\lambda x)$, and the solution
for that equation is $y_{\mathsf{s}}(x)=1-\exp(-\ln(1+\lambda x)/\lambda)$.
Thus
$\gir(\sbgw{\lambda})
=y_{\mathsf{s}}(1)
=1-(1+\lambda)^{-1/\lambda}
=1-e^{-\gir(\gw{\lambda})}$.
In particular, $\gir(\sbgw{1}) = 1/2$.

\paragraph{Uniform spanning trees}
It is a classical (and beautiful) fact (see, e.g.,~\cites{Kol77,Gri80})
that if $T_n$ is a uniformly chosen random tree drawn from the set of $n^{n-2}$ trees on (labelled) $n$ vertices, then $T_n$ converges locally to $\sbgw{1}$, hence $\gir(T_n)\sim 1/2$ \whp{}.
To the best of our knowledge, this intriguing fact was not previously known.
Recently, after a conference version of this paper was published, Contat~\cite{Con22} proved a much stronger statement concerning the cardinality of the random greedy independent set in uniform random trees, showing that it has essentially the same law as its complement.
In a newer version of her paper, she obtained the exact distribution of the size of the random greedy independent set, showing that it has the same distribution as the number of vertices at even height in a uniformly sampled rooted random tree.
The exact distribution was also obtained, independently, by Panholzer~\cite{Pan20}.

Nachmias and Peres~\cite{NP22} showed (see also~\cite{HNT18}) that if $G_n$ is a sequence of finite, simple, connected regular graphs with degree tending to infinity,
and $T_n$ is the uniform spanning tree of $G_n$,
then $T_n$ converges locally to $\sbgw{1}$.
It follows that $\gir(T_n)\sim 1/2$ \whp{} in this case as well.

\paragraph{Random functional digraphs} 
It can be easily verified that the local limit of a random functional digraph $\vec{G}_1(n)$ (the digraph on $n$ vertices whose edges are $(i,\pi(i))$ for a uniform random permutation $\pi$), with orientations ignored, is also $\sbgw{1}$, hence $\gir(\vec{G}_1)\sim 1/2$ \whp{}.

\paragraph{Sparse random planar graphs}
\newcommand{\cP}{\mathcal{P}}
Let $\cP(n,\lambda)$ denote the uniform distribution over the set of (labelled) planar graphs on $n$ vertices with $\lambda n/2$ edges.
According to a recent result by Kang and Missethan~\cite{KM21+}, if $\lambda\in(0,1]$ then $\cP(n,\lambda)$ converges locally to $\gw{\lambda}$, hence $\gir(\cP(n,\lambda))\sim\ln(1+\lambda)/\lambda$ \whp{}; and if $\lambda\in(1,2]$ then $\cP(n,\lambda)$ converges locally to $(\lambda-1)\sbgw{1}+(2-\lambda)\gw{1}$, namely, to the random tree which is sampled from $\sbgw{1}$ with probability $\lambda-1$ and from $\gw{1}$ with probability $2-\lambda$.
It follows that in this case, $\gir(\cP(n,\lambda))\sim(\lambda-1)/2+(2-\lambda)\ln{2}$ \whp{}.
Note that $\gir(\cP(n,\lambda))$ is continuous for $\lambda\in(0,2]$.

\subsection{\texorpdfstring{$d$}{d}-ary Trees}\label{sec:ary}
For $d>1$, let $\ary{d}$ be the $d$-ary tree.  It may be viewed as a (single
type) branching process.  It thus immediately follows from \eqref{eq:p:ary} that
$y'(x)=(1-y(x))^d$.
The solution for this differential equation is $y(x)=1-((d-1)x+1)^{-1/(d-1)}$.
It follows that $\gir(\ary{d}) = y(1) = 1-d^{-1/(d-1)}$.
This fact also follows easily using the generating functions approach described
in \cref{sec:pgf}.
A remarkable example is $\gir(\ary{2}) = 1/2$.

\subsection{Regular Trees}\label{sec:reg}
For $d\ge 3$, let $\reg{d}$ be the $d$-regular tree.  It may viewed as a
two-type branching process with types $d$ for the root and $d-1$ for the rest of
the vertices.
The fundamental system of ODEs in this case is $y_d'(x)=(1-y_{d-1}(x))^d$,
and from \cref{sec:ary} we obtain $y_{d-1}(x)=1-((d-2)x+1)^{-1/(d-2)}$.
It follows that $y_d'(x) = ((d-2)x+1)^{-d/(d-2)}$, of which the solution is
$y_d(x) = (1 - ((d-2)x+1)^{-2/(d-2)})/2$.
Therefore,
\begin{equation}\label{eq:regtree}
\gir(\reg{d}) = y_d(1) = \frac{1}{2}\left(1-(d-1)^{-2/(d-2)}\right).
\end{equation}
We remark that a similar derivation of \eqref{eq:regtree} was obtained, using a similar method, by Penrose and Sudbury~\cite{PS05}, and was derived earlier by Fan and Percus~\cite{FP91}.
In both works, however, the application for random regular graphs (see below) is absent.

As with $d$-ary trees, here again the generating functions approach works
easily: the solution to $\int_{h(x)}^1 z^{d-1}dz=x$ is $h(x)=(1-(2-d)x)^{1/(2-d)}$, and the result follows from \cref{cl:pgf:iid}.
Remarkable examples include $\gir(\reg{3}) = 3/8$ and $\gir(\reg{4}) = 1/3$.

\paragraph{Random regular graphs}
Since the random regular graph $G(n,d)$ (a uniformly sampled graph from the set of all $d$-regular graphs on $n$ vertices, assuming $dn$ is even) converges locally to $\reg{d}$ (see, e.g.,~\cite{Wor99RRG}),
the above result for this case is exactly~\cite{Wor95}*{Theorem 4}.
In fact, since any sequence of $d$-regular graphs with girth tending to infinity converges locally to $\reg{d}$, we also recover \cite{LW07}*{Theorem 2}.
This latter result was proved later, using different methods, by Gamarnik and Goldberg~\cite{GG10}.

\section{Hypergraphs}\label{sec:hyper}
\begin{figure}
  \captionsetup{width=0.879\textwidth,font=small}
  \centering
\begin{tikzpicture}[svertex/.style={fill,circle},edge/.style={color=brown,fill,fill opacity=0.1}]
    \node[svertex,inner sep=1.50pt] (0) at (0,0) {};
    \node[svertex,inner sep=1.00pt] (0ua) at ($(0) + (90.0:1.00)$) {};
    \draw[edge,rotate around={90:(0ua)}] (0ua) ellipse (1.40 and 0.20);
    \node[svertex,inner sep=1.00pt] (0ub) at ($(0) + (90.0:2.00)$) {};
    \node[svertex,inner sep=1.00pt] (0la) at ($(0) + (-30.0:1.00)$) {};
    \draw[edge,rotate around={-30:(0la)}] (0la) ellipse (1.40 and 0.20);
    \node[svertex,inner sep=1.00pt] (0lb) at ($(0) + (-30.0:2.00)$) {};
    \node[svertex,inner sep=1.00pt] (0ra) at ($(0) + (-150.0:1.00)$) {};
    \draw[edge,rotate around={-150:(0ra)}] (0ra) ellipse (1.40 and 0.20);
    \node[svertex,inner sep=1.00pt] (0rb) at ($(0) + (-150.0:2.00)$) {};
    \node[svertex,inner sep=0.50pt] (0uala) at ($(0ua) + (165.0:0.50)$) {};
    \draw[edge,rotate around={165:(0uala)}] (0uala) ellipse (0.70 and 0.14);
    \node[svertex,inner sep=0.50pt] (0ualb) at ($(0ua) + (165.0:1.00)$) {};
    \node[svertex,inner sep=0.50pt] (0uara) at ($(0ua) + (15.0:0.50)$) {};
    \draw[edge,rotate around={15:(0uara)}] (0uara) ellipse (0.70 and 0.14);
    \node[svertex,inner sep=0.50pt] (0uarb) at ($(0ua) + (15.0:1.00)$) {};
    \node[svertex,inner sep=0.50pt] (0ubla) at ($(0ub) + (150.0:0.50)$) {};
    \draw[edge,rotate around={150:(0ubla)}] (0ubla) ellipse (0.70 and 0.14);
    \node[svertex,inner sep=0.50pt] (0ublb) at ($(0ub) + (150.0:1.00)$) {};
    \node[svertex,inner sep=0.50pt] (0ubra) at ($(0ub) + (30.0:0.50)$) {};
    \draw[edge,rotate around={30:(0ubra)}] (0ubra) ellipse (0.70 and 0.14);
    \node[svertex,inner sep=0.50pt] (0ubrb) at ($(0ub) + (30.0:1.00)$) {};
    \node[svertex,inner sep=0.50pt] (0lala) at ($(0la) + (45.0:0.50)$) {};
    \draw[edge,rotate around={45:(0lala)}] (0lala) ellipse (0.70 and 0.14);
    \node[svertex,inner sep=0.50pt] (0lalb) at ($(0la) + (45.0:1.00)$) {};
    \node[svertex,inner sep=0.50pt] (0lara) at ($(0la) + (-105.0:0.50)$) {};
    \draw[edge,rotate around={-105:(0lara)}] (0lara) ellipse (0.70 and 0.14);
    \node[svertex,inner sep=0.50pt] (0larb) at ($(0la) + (-105.0:1.00)$) {};
    \node[svertex,inner sep=0.50pt] (0lbla) at ($(0lb) + (30.0:0.50)$) {};
    \draw[edge,rotate around={30:(0lbla)}] (0lbla) ellipse (0.70 and 0.14);
    \node[svertex,inner sep=0.50pt] (0lblb) at ($(0lb) + (30.0:1.00)$) {};
    \node[svertex,inner sep=0.50pt] (0lbra) at ($(0lb) + (-90.0:0.50)$) {};
    \draw[edge,rotate around={-90:(0lbra)}] (0lbra) ellipse (0.70 and 0.14);
    \node[svertex,inner sep=0.50pt] (0lbrb) at ($(0lb) + (-90.0:1.00)$) {};
    \node[svertex,inner sep=0.50pt] (0rala) at ($(0ra) + (-75.0:0.50)$) {};
    \draw[edge,rotate around={-75:(0rala)}] (0rala) ellipse (0.70 and 0.14);
    \node[svertex,inner sep=0.50pt] (0ralb) at ($(0ra) + (-75.0:1.00)$) {};
    \node[svertex,inner sep=0.50pt] (0rara) at ($(0ra) + (-225.0:0.50)$) {};
    \draw[edge,rotate around={-225:(0rara)}] (0rara) ellipse (0.70 and 0.14);
    \node[svertex,inner sep=0.50pt] (0rarb) at ($(0ra) + (-225.0:1.00)$) {};
    \node[svertex,inner sep=0.50pt] (0rbla) at ($(0rb) + (-90.0:0.50)$) {};
    \draw[edge,rotate around={-90:(0rbla)}] (0rbla) ellipse (0.70 and 0.14);
    \node[svertex,inner sep=0.50pt] (0rblb) at ($(0rb) + (-90.0:1.00)$) {};
    \node[svertex,inner sep=0.50pt] (0rbra) at ($(0rb) + (-210.0:0.50)$) {};
    \draw[edge,rotate around={-210:(0rbra)}] (0rbra) ellipse (0.70 and 0.14);
    \node[svertex,inner sep=0.50pt] (0rbrb) at ($(0rb) + (-210.0:1.00)$) {};
    \node[svertex,inner sep=0.20pt] (0ualala) at ($(0uala) + (240.0:0.20)$) {};
    \draw[edge,rotate around={240:(0ualala)}] (0ualala) ellipse (0.28 and 0.09);
    \node[svertex,inner sep=0.20pt] (0ualalb) at ($(0uala) + (240.0:0.40)$) {};
    \node[svertex,inner sep=0.20pt] (0ualara) at ($(0uala) + (90.0:0.20)$) {};
    \draw[edge,rotate around={90:(0ualara)}] (0ualara) ellipse (0.28 and 0.09);
    \node[svertex,inner sep=0.20pt] (0ualarb) at ($(0uala) + (90.0:0.40)$) {};
    \node[svertex,inner sep=0.20pt] (0ualbla) at ($(0ualb) + (225.0:0.20)$) {};
    \draw[edge,rotate around={225:(0ualbla)}] (0ualbla) ellipse (0.28 and 0.09);
    \node[svertex,inner sep=0.20pt] (0ualblb) at ($(0ualb) + (225.0:0.40)$) {};
    \node[svertex,inner sep=0.20pt] (0ualbra) at ($(0ualb) + (105.0:0.20)$) {};
    \draw[edge,rotate around={105:(0ualbra)}] (0ualbra) ellipse (0.28 and 0.09);
    \node[svertex,inner sep=0.20pt] (0ualbrb) at ($(0ualb) + (105.0:0.40)$) {};
    \node[svertex,inner sep=0.20pt] (0uarala) at ($(0uara) + (90.0:0.20)$) {};
    \draw[edge,rotate around={90:(0uarala)}] (0uarala) ellipse (0.28 and 0.09);
    \node[svertex,inner sep=0.20pt] (0uaralb) at ($(0uara) + (90.0:0.40)$) {};
    \node[svertex,inner sep=0.20pt] (0uarara) at ($(0uara) + (-60.0:0.20)$) {};
    \draw[edge,rotate around={-60:(0uarara)}] (0uarara) ellipse (0.28 and 0.09);
    \node[svertex,inner sep=0.20pt] (0uararb) at ($(0uara) + (-60.0:0.40)$) {};
    \node[svertex,inner sep=0.20pt] (0uarbla) at ($(0uarb) + (75.0:0.20)$) {};
    \draw[edge,rotate around={75:(0uarbla)}] (0uarbla) ellipse (0.28 and 0.09);
    \node[svertex,inner sep=0.20pt] (0uarblb) at ($(0uarb) + (75.0:0.40)$) {};
    \node[svertex,inner sep=0.20pt] (0uarbra) at ($(0uarb) + (-45.0:0.20)$) {};
    \draw[edge,rotate around={-45:(0uarbra)}] (0uarbra) ellipse (0.28 and 0.09);
    \node[svertex,inner sep=0.20pt] (0uarbrb) at ($(0uarb) + (-45.0:0.40)$) {};
    \node[svertex,inner sep=0.20pt] (0ublala) at ($(0ubla) + (225.0:0.20)$) {};
    \draw[edge,rotate around={225:(0ublala)}] (0ublala) ellipse (0.28 and 0.09);
    \node[svertex,inner sep=0.20pt] (0ublalb) at ($(0ubla) + (225.0:0.40)$) {};
    \node[svertex,inner sep=0.20pt] (0ublara) at ($(0ubla) + (75.0:0.20)$) {};
    \draw[edge,rotate around={75:(0ublara)}] (0ublara) ellipse (0.28 and 0.09);
    \node[svertex,inner sep=0.20pt] (0ublarb) at ($(0ubla) + (75.0:0.40)$) {};
    \node[svertex,inner sep=0.20pt] (0ublbla) at ($(0ublb) + (210.0:0.20)$) {};
    \draw[edge,rotate around={210:(0ublbla)}] (0ublbla) ellipse (0.28 and 0.09);
    \node[svertex,inner sep=0.20pt] (0ublblb) at ($(0ublb) + (210.0:0.40)$) {};
    \node[svertex,inner sep=0.20pt] (0ublbra) at ($(0ublb) + (90.0:0.20)$) {};
    \draw[edge,rotate around={90:(0ublbra)}] (0ublbra) ellipse (0.28 and 0.09);
    \node[svertex,inner sep=0.20pt] (0ublbrb) at ($(0ublb) + (90.0:0.40)$) {};
    \node[svertex,inner sep=0.20pt] (0ubrala) at ($(0ubra) + (105.0:0.20)$) {};
    \draw[edge,rotate around={105:(0ubrala)}] (0ubrala) ellipse (0.28 and 0.09);
    \node[svertex,inner sep=0.20pt] (0ubralb) at ($(0ubra) + (105.0:0.40)$) {};
    \node[svertex,inner sep=0.20pt] (0ubrara) at ($(0ubra) + (-45.0:0.20)$) {};
    \draw[edge,rotate around={-45:(0ubrara)}] (0ubrara) ellipse (0.28 and 0.09);
    \node[svertex,inner sep=0.20pt] (0ubrarb) at ($(0ubra) + (-45.0:0.40)$) {};
    \node[svertex,inner sep=0.20pt] (0ubrbla) at ($(0ubrb) + (90.0:0.20)$) {};
    \draw[edge,rotate around={90:(0ubrbla)}] (0ubrbla) ellipse (0.28 and 0.09);
    \node[svertex,inner sep=0.20pt] (0ubrblb) at ($(0ubrb) + (90.0:0.40)$) {};
    \node[svertex,inner sep=0.20pt] (0ubrbra) at ($(0ubrb) + (-30.0:0.20)$) {};
    \draw[edge,rotate around={-30:(0ubrbra)}] (0ubrbra) ellipse (0.28 and 0.09);
    \node[svertex,inner sep=0.20pt] (0ubrbrb) at ($(0ubrb) + (-30.0:0.40)$) {};
    \node[svertex,inner sep=0.20pt] (0lalala) at ($(0lala) + (120.0:0.20)$) {};
    \draw[edge,rotate around={120:(0lalala)}] (0lalala) ellipse (0.28 and 0.09);
    \node[svertex,inner sep=0.20pt] (0lalalb) at ($(0lala) + (120.0:0.40)$) {};
    \node[svertex,inner sep=0.20pt] (0lalara) at ($(0lala) + (-30.0:0.20)$) {};
    \draw[edge,rotate around={-30:(0lalara)}] (0lalara) ellipse (0.28 and 0.09);
    \node[svertex,inner sep=0.20pt] (0lalarb) at ($(0lala) + (-30.0:0.40)$) {};
    \node[svertex,inner sep=0.20pt] (0lalbla) at ($(0lalb) + (105.0:0.20)$) {};
    \draw[edge,rotate around={105:(0lalbla)}] (0lalbla) ellipse (0.28 and 0.09);
    \node[svertex,inner sep=0.20pt] (0lalblb) at ($(0lalb) + (105.0:0.40)$) {};
    \node[svertex,inner sep=0.20pt] (0lalbra) at ($(0lalb) + (-15.0:0.20)$) {};
    \draw[edge,rotate around={-15:(0lalbra)}] (0lalbra) ellipse (0.28 and 0.09);
    \node[svertex,inner sep=0.20pt] (0lalbrb) at ($(0lalb) + (-15.0:0.40)$) {};
    \node[svertex,inner sep=0.20pt] (0larala) at ($(0lara) + (-30.0:0.20)$) {};
    \draw[edge,rotate around={-30:(0larala)}] (0larala) ellipse (0.28 and 0.09);
    \node[svertex,inner sep=0.20pt] (0laralb) at ($(0lara) + (-30.0:0.40)$) {};
    \node[svertex,inner sep=0.20pt] (0larara) at ($(0lara) + (-180.0:0.20)$) {};
    \draw[edge,rotate around={-180:(0larara)}] (0larara) ellipse (0.28 and 0.09);
    \node[svertex,inner sep=0.20pt] (0lararb) at ($(0lara) + (-180.0:0.40)$) {};
    \node[svertex,inner sep=0.20pt] (0larbla) at ($(0larb) + (-45.0:0.20)$) {};
    \draw[edge,rotate around={-45:(0larbla)}] (0larbla) ellipse (0.28 and 0.09);
    \node[svertex,inner sep=0.20pt] (0larblb) at ($(0larb) + (-45.0:0.40)$) {};
    \node[svertex,inner sep=0.20pt] (0larbra) at ($(0larb) + (-165.0:0.20)$) {};
    \draw[edge,rotate around={-165:(0larbra)}] (0larbra) ellipse (0.28 and 0.09);
    \node[svertex,inner sep=0.20pt] (0larbrb) at ($(0larb) + (-165.0:0.40)$) {};
    \node[svertex,inner sep=0.20pt] (0lblala) at ($(0lbla) + (105.0:0.20)$) {};
    \draw[edge,rotate around={105:(0lblala)}] (0lblala) ellipse (0.28 and 0.09);
    \node[svertex,inner sep=0.20pt] (0lblalb) at ($(0lbla) + (105.0:0.40)$) {};
    \node[svertex,inner sep=0.20pt] (0lblara) at ($(0lbla) + (-45.0:0.20)$) {};
    \draw[edge,rotate around={-45:(0lblara)}] (0lblara) ellipse (0.28 and 0.09);
    \node[svertex,inner sep=0.20pt] (0lblarb) at ($(0lbla) + (-45.0:0.40)$) {};
    \node[svertex,inner sep=0.20pt] (0lblbla) at ($(0lblb) + (90.0:0.20)$) {};
    \draw[edge,rotate around={90:(0lblbla)}] (0lblbla) ellipse (0.28 and 0.09);
    \node[svertex,inner sep=0.20pt] (0lblblb) at ($(0lblb) + (90.0:0.40)$) {};
    \node[svertex,inner sep=0.20pt] (0lblbra) at ($(0lblb) + (-30.0:0.20)$) {};
    \draw[edge,rotate around={-30:(0lblbra)}] (0lblbra) ellipse (0.28 and 0.09);
    \node[svertex,inner sep=0.20pt] (0lblbrb) at ($(0lblb) + (-30.0:0.40)$) {};
    \node[svertex,inner sep=0.20pt] (0lbrala) at ($(0lbra) + (-15.0:0.20)$) {};
    \draw[edge,rotate around={-15:(0lbrala)}] (0lbrala) ellipse (0.28 and 0.09);
    \node[svertex,inner sep=0.20pt] (0lbralb) at ($(0lbra) + (-15.0:0.40)$) {};
    \node[svertex,inner sep=0.20pt] (0lbrara) at ($(0lbra) + (-165.0:0.20)$) {};
    \draw[edge,rotate around={-165:(0lbrara)}] (0lbrara) ellipse (0.28 and 0.09);
    \node[svertex,inner sep=0.20pt] (0lbrarb) at ($(0lbra) + (-165.0:0.40)$) {};
    \node[svertex,inner sep=0.20pt] (0lbrbla) at ($(0lbrb) + (-30.0:0.20)$) {};
    \draw[edge,rotate around={-30:(0lbrbla)}] (0lbrbla) ellipse (0.28 and 0.09);
    \node[svertex,inner sep=0.20pt] (0lbrblb) at ($(0lbrb) + (-30.0:0.40)$) {};
    \node[svertex,inner sep=0.20pt] (0lbrbra) at ($(0lbrb) + (-150.0:0.20)$) {};
    \draw[edge,rotate around={-150:(0lbrbra)}] (0lbrbra) ellipse (0.28 and 0.09);
    \node[svertex,inner sep=0.20pt] (0lbrbrb) at ($(0lbrb) + (-150.0:0.40)$) {};
    \node[svertex,inner sep=0.20pt] (0ralala) at ($(0rala) + (0.0:0.20)$) {};
    \draw[edge,rotate around={0:(0ralala)}] (0ralala) ellipse (0.28 and 0.09);
    \node[svertex,inner sep=0.20pt] (0ralalb) at ($(0rala) + (0.0:0.40)$) {};
    \node[svertex,inner sep=0.20pt] (0ralara) at ($(0rala) + (-150.0:0.20)$) {};
    \draw[edge,rotate around={-150:(0ralara)}] (0ralara) ellipse (0.28 and 0.09);
    \node[svertex,inner sep=0.20pt] (0ralarb) at ($(0rala) + (-150.0:0.40)$) {};
    \node[svertex,inner sep=0.20pt] (0ralbla) at ($(0ralb) + (-15.0:0.20)$) {};
    \draw[edge,rotate around={-15:(0ralbla)}] (0ralbla) ellipse (0.28 and 0.09);
    \node[svertex,inner sep=0.20pt] (0ralblb) at ($(0ralb) + (-15.0:0.40)$) {};
    \node[svertex,inner sep=0.20pt] (0ralbra) at ($(0ralb) + (-135.0:0.20)$) {};
    \draw[edge,rotate around={-135:(0ralbra)}] (0ralbra) ellipse (0.28 and 0.09);
    \node[svertex,inner sep=0.20pt] (0ralbrb) at ($(0ralb) + (-135.0:0.40)$) {};
    \node[svertex,inner sep=0.20pt] (0rarala) at ($(0rara) + (-150.0:0.20)$) {};
    \draw[edge,rotate around={-150:(0rarala)}] (0rarala) ellipse (0.28 and 0.09);
    \node[svertex,inner sep=0.20pt] (0raralb) at ($(0rara) + (-150.0:0.40)$) {};
    \node[svertex,inner sep=0.20pt] (0rarara) at ($(0rara) + (-300.0:0.20)$) {};
    \draw[edge,rotate around={-300:(0rarara)}] (0rarara) ellipse (0.28 and 0.09);
    \node[svertex,inner sep=0.20pt] (0rararb) at ($(0rara) + (-300.0:0.40)$) {};
    \node[svertex,inner sep=0.20pt] (0rarbla) at ($(0rarb) + (-165.0:0.20)$) {};
    \draw[edge,rotate around={-165:(0rarbla)}] (0rarbla) ellipse (0.28 and 0.09);
    \node[svertex,inner sep=0.20pt] (0rarblb) at ($(0rarb) + (-165.0:0.40)$) {};
    \node[svertex,inner sep=0.20pt] (0rarbra) at ($(0rarb) + (-285.0:0.20)$) {};
    \draw[edge,rotate around={-285:(0rarbra)}] (0rarbra) ellipse (0.28 and 0.09);
    \node[svertex,inner sep=0.20pt] (0rarbrb) at ($(0rarb) + (-285.0:0.40)$) {};
    \node[svertex,inner sep=0.20pt] (0rblala) at ($(0rbla) + (-15.0:0.20)$) {};
    \draw[edge,rotate around={-15:(0rblala)}] (0rblala) ellipse (0.28 and 0.09);
    \node[svertex,inner sep=0.20pt] (0rblalb) at ($(0rbla) + (-15.0:0.40)$) {};
    \node[svertex,inner sep=0.20pt] (0rblara) at ($(0rbla) + (-165.0:0.20)$) {};
    \draw[edge,rotate around={-165:(0rblara)}] (0rblara) ellipse (0.28 and 0.09);
    \node[svertex,inner sep=0.20pt] (0rblarb) at ($(0rbla) + (-165.0:0.40)$) {};
    \node[svertex,inner sep=0.20pt] (0rblbla) at ($(0rblb) + (-30.0:0.20)$) {};
    \draw[edge,rotate around={-30:(0rblbla)}] (0rblbla) ellipse (0.28 and 0.09);
    \node[svertex,inner sep=0.20pt] (0rblblb) at ($(0rblb) + (-30.0:0.40)$) {};
    \node[svertex,inner sep=0.20pt] (0rblbra) at ($(0rblb) + (-150.0:0.20)$) {};
    \draw[edge,rotate around={-150:(0rblbra)}] (0rblbra) ellipse (0.28 and 0.09);
    \node[svertex,inner sep=0.20pt] (0rblbrb) at ($(0rblb) + (-150.0:0.40)$) {};
    \node[svertex,inner sep=0.20pt] (0rbrala) at ($(0rbra) + (-135.0:0.20)$) {};
    \draw[edge,rotate around={-135:(0rbrala)}] (0rbrala) ellipse (0.28 and 0.09);
    \node[svertex,inner sep=0.20pt] (0rbralb) at ($(0rbra) + (-135.0:0.40)$) {};
    \node[svertex,inner sep=0.20pt] (0rbrara) at ($(0rbra) + (-285.0:0.20)$) {};
    \draw[edge,rotate around={-285:(0rbrara)}] (0rbrara) ellipse (0.28 and 0.09);
    \node[svertex,inner sep=0.20pt] (0rbrarb) at ($(0rbra) + (-285.0:0.40)$) {};
    \node[svertex,inner sep=0.20pt] (0rbrbla) at ($(0rbrb) + (-150.0:0.20)$) {};
    \draw[edge,rotate around={-150:(0rbrbla)}] (0rbrbla) ellipse (0.28 and 0.09);
    \node[svertex,inner sep=0.20pt] (0rbrblb) at ($(0rbrb) + (-150.0:0.40)$) {};
    \node[svertex,inner sep=0.20pt] (0rbrbra) at ($(0rbrb) + (-270.0:0.20)$) {};
    \draw[edge,rotate around={-270:(0rbrbra)}] (0rbrbra) ellipse (0.28 and 0.09);
    \node[svertex,inner sep=0.20pt] (0rbrbrb) at ($(0rbrb) + (-270.0:0.40)$) {};
\end{tikzpicture}
  \caption{The ball of radius $3$ around the root of the infinite $3$-uniform $3$-regular loose hypertree $\reg{3}^3$.}
  \label{fig:hypertree}
\end{figure}
Recently (after a conference version of this paper was posted), Nie and Verstra\"ete~\cite{NV21} analysed the random greedy algorithm for producing maximal independent sets in $r$-uniform $d$-regular high-girth (linear) hypergraphs.
In this section we show how to deduce their main results in our framework, assuming the girth tends to infinity.

It is not hard to check that the local limit of $r$-uniform $d$-regular hypergraphs with girth tending to infinity is the rooted infinite $r$-uniform $d$-regular loose hypertree (see \cref{fig:hypertree}), denoted here by $\reg{d}^r$ (this is just the hypergraph-analogue to the graph case described in \cref{sec:reg}).
This hypertree may be viewed as a two-type branching process with types $d$ for the root and $d-1$ for the rest of the vertices, where a type $d$ vertex has $d$ incident edges, on each there are $r-1$ additional vertices of type $d-1$, and a type $d-1$ vertex has $d-1$ incident edges, on each there are $r-1$ additional vertices of type $d-1$.
The fundamental system of ODEs in this case is thus
\begin{align}
  y'_d(x) &= (1-y_{d-1}^{r-1}(x))^d, \label{eq:hyper:1}\\
  y'_{d-1}(x) &= (1-y_{d-1}^{r-1}(x))^{d-1}. \label{eq:hyper:2}
\end{align}
The second equation is separable, thus
\[
\begin{aligned}
  x &= \int dx
    = \int \frac{y'_{d-1}(x)}{\left(1-y_{d-1}^{r-1}(x)\right)^{d-1}}dx
    = \int \frac{dy}{\left(1-y^{r-1}\right)^{d-1}}\\
    &= \int \sum_{n=0}^\infty \binom{n+d-2}{d-2} y^{n(r-1)}dy\\
    &= \sum_{n=0}^\infty \binom{n+d-2}{d-2} \int y^{n(r-1)}dy
    = \sum_{n=0}^\infty \binom{n+d-2}{d-2} \frac{y^{n(r-1)+1}}{n(r-1)+1} + c.
\end{aligned}
\]
Considering the initial conditions, we obtain $c=0$.
Setting
\[
  H(y) = \sum_{n=0}^\infty \binom{n+d-2}{d-2} \frac{y^{n(r-1)+1}}{n(r-1)+1},
\]
the solution to~\eqref{eq:hyper:2} is $y_{d-1}(x)=H^{-1}(x)$.
As for~\eqref{eq:hyper:1}, set $u=H^{-1}(t)$ and then (by~\eqref{eq:hyper:2}),
$du=(1-u^{r-1})^{d-1}dt$, hence
\[
\begin{aligned}
  y_d(x) &= \int_0^x \left(1-y_{d-1}^{r-1}(t)\right)^d dt
         = \int_0^x \left(1-(H^{-1}(t))^{r-1}\right)^d dt\\
         &= \int_{H^{-1}(0)}^{H^{-1}(x)} \frac{\left(1-u^{r-1}\right)^d}{\left(1-u^{r-1}\right)^{d-1}} du\\
         &= \int_0^{H^{-1}(x)} (1-u^{r-1})du
         = H^{-1}(x)-\frac{\left(H^{-1}(x)\right)^r}{r}.
\end{aligned}
\]
In particular, $\iota(\reg{d}^r)=y_d(1) = H^{-1}(1)-(H^{-1}(1))^r/r$, and together with a hypergraph analogue of \cref{thm:mean} this recovers~\cite{NV21}*{Theorem 4}, up to the error bounds.
Using (a hypergraph analogue of) \cref{thm:main} we also obtain a nonquantitative (namely, without an explicit bound on the deviation from the mean) version of \cite{NV21}*{Theorem 5}.
The advantage of the above method is, however, its greater generality, as it allows for a wider range of assumptions on the underlying hypergraph.

\section{Lower bound in trees}\label{sec:trees}
Let us focus on trees.  How large can the expected greedy independent ratio be?  How small can it be?
The sequence of stars is a clear witness that the only possible asymptotic upper bound
is the trivial one, namely $1$.
In fact, the $n$-vertex star is the {\em unique} maximiser among $n$-vertex trees (see below).
Apparently, the lower bound is not trivial.
An immediate corollary of \cref{thm:trees,thm:mean} is that a tight asymptotic lower bound is $\gir(\ZZ)=(1-e^{-2})/2$
(compare with \cite{Sud09}).
The statement of \cref{thm:trees} is, however, much stronger: paths achieve the \emph{exact} (non-asymptotic) lower bound for the expected greedy independence ratio among the set of all trees of a given order.
It is reasonable to expect that the path is the {\em unique} minimiser;
our proof of \cref{thm:trees} does not imply that (see \cref{sec:open}).

To prove \cref{thm:trees} we will need to first gain deeper understanding of the behaviour of the greedy algorithm on the path.
For a graph $G$ denote by $\gic(G)$ the cardinality of its greedy independent set, and let $\egic(G)=\E[\gic(G)]$.
Let $\alpha_n=\egic(P_n)$.
Suppose the vertices of $P_n$ are $1,\ldots,n$,
and let $s$ be the vertex which is first in the permutation of the vertices.
Setting $\alpha_{-1}=\alpha_0=0$, we obtain the recursion
\begin{equation}\label{eq:path}
  \alpha_n = \E[\E[\gic(P_n)\mid s]]
  = \frac{1}{n}\sum_{i=1}^n (1+\alpha_{i-2}+\alpha_{n-i-1})
  = 1 + \frac{2}{n}\sum_{i=1}^n \alpha_{i-2}.
\end{equation}
The following explicit formula for $\alpha_n$ ($n\ge 0$) appears in \cite{FS62}:
\begin{equation}\label{eq:alpha:explicit}
\alpha_n = \sum_{i=0}^{n-1}\frac{(-2)^i(n-i)}{(i+1)!}.
\end{equation}
The main properties of $\alpha_n$ that we need in this section are given by the following two lemmas.
We defer their (somewhat technical) proofs to \cref{sec:calc}.

\begin{lemma}\label{lem:monotone:subadd}
  The sequence $\alpha_n$ is monotone increasing and subadditive.
\end{lemma}
A natural approach for trying to prove \cref{lem:monotone:subadd} is
by using the intuitive assertion that $\egic{}$ is monotone with respect to edge deletion.
This is unfortunately false;
indeed, let $S_n$ denote the star with $n$ leaves.
It is easy to verify that $\egic(S_n)=\frac{1}{n+1}\cdot 1 + \frac{n}{n+1}\cdot n =(n^2+1)/(n+1)$.
Let $T$ be obtained by taking two copies of $S_n$
and joining their centres by an edge $e$.
One can check that $\egic(T) = \frac{1}{n+1}\cdot(1+n) + \frac{n}{n+1}\cdot(n+\egic(S_n))
=(2n^3+2n^2+3n+1)/(n+1)^2$,
while $\egic(T-e) = 2\egic(S_n) = 2(n^2+1)/(n+1)$, which is strictly smaller for every $n\ge 2$.
We deal with the difficulty illustrated by this counterintuitive example by using a more involved argument;
see \cref{sec:calc} below.

Our analysis also relies on the following technical fact.
Define \[\xi_{n,\ell}=\sum_{j=1}^\ell \alpha_{n+j}.\]
\begin{lemma}\label{lem:xi}
  For every $\ell,a,b\ge 1$ it holds that
  $\xi_{a,\ell} + \xi_{b,\ell} \le \xi_{a+b,\ell} + \xi_{0,\ell}$.
\end{lemma}

Before we present the main tools to be used in the proof of \cref{thm:trees},
let us show that the $n$-vertex star is the unique maximiser of $\egic$.
Indeed, the independence number $\alpha(T)$ of any $n$-vertex tree $T$ that is not a star
is at most $n-2$. Thus,
\[
  \egic(S_{n-1}) = \frac{1}{n}((n-1)^2+1)
  = n-2+\frac{2}{n}
  > n-2 \ge \alpha(T) \ge \egic(T).
\]

\subsection{KC-Transformations}\label{sec:KC}
\tikzstyle{level 1}=[level distance=4.5mm,sibling angle=75]
\tikzstyle{level 2}=[level distance=3.5mm,line width=0.08ex,sibling angle=56]
\tikzstyle{level 3}=[level distance=2.75mm,line width=0.06ex,sibling angle=52]
\tikzstyle{level 4}=[level distance=2mm,line width=0.04ex,sibling angle=48]
\tikzstyle{level 5}=[level distance=1mm,line width=0.02ex,sibling angle=44]
\tikzstyle{level 6}=[level distance=1.75mm,line width=0.01ex,sibling angle=40]
\tikzstyle{level 1}=[level distance=6mm,sibling angle=90]
\tikzstyle{level 2}=[level distance=3.5mm,line width=0.08ex,sibling angle=70]
\tikzstyle{level 3}=[level distance=2.75mm,line width=0.06ex,sibling angle=50]
\tikzstyle{level 4}=[level distance=2mm,line width=0.04ex,sibling angle=40]
\tikzstyle{level 5}=[level distance=1mm,line width=0.02ex,sibling angle=30]
\tikzstyle{level 6}=[level distance=1.75mm,line width=0.01ex,sibling angle=20]
\begin{figure}
  \captionsetup{width=0.879\textwidth,font=small}
  \centering
  \begin{tikzpicture}[
    svertex/.style={fill,circle,inner sep=1.5pt},
    scale=1.5,
  ]
  \begin{scope}
  \node[svertex,label={[label distance=1.25em,anchor=north]90:$x$}] (x) at (0,0) {} [grow cyclic,rotate=180]
  child foreach \a in {0,1} {
    child foreach \b in {0,1} {
      child foreach \c in {0,1} {
        child foreach \d in {0,1} {
          child foreach \e in {0,1} {
            child foreach \leafcolor in {red, orange}
              { edge from parent [color=\leafcolor] }
          }
        }
      }
    } edge from parent
  };
  \node[svertex,label={[label distance=1.25em,anchor=north]90:$y$}] (y) at (1,0) {} [grow cyclic,rotate=0]
  child foreach \a in {0,1} {
    child foreach \b in {0,1} {
      child foreach \c in {0,1} {
        child foreach \d in {0,1} {
          child foreach \e in {0,1} {
            child foreach \leafcolor in {red, orange}
              { edge from parent [color=\leafcolor] }
          }
        }
      }
    } edge from parent
  };
  \draw (x) -- (y);
  \foreach \x in {0.167,0.333,...,0.833} {
    \node[fill,color=gray,circle,inner sep=0.75pt] at (\x,0) {};
  }
  \end{scope}
  \begin{scope}[xshift=3.5cm]
  \node (to) at (0,0) {$\implies$};
  \end{scope}
  \begin{scope}[xshift=6cm]
  \coordinate (_y) at (0,0) {} [grow cyclic,rotate=0]
  child foreach \a in {0,1} {
    child foreach \b in {0,1} {
      child foreach \c in {0,1} {
        child foreach \d in {0,1} {
          child foreach \e in {0,1} {
            child foreach \leafcolor in {red, orange}
              { edge from parent [color=\leafcolor] }
          }
        }
      }
    } edge from parent
  };
  \node[svertex,label={270:$x$}] (x) at (0,0) {} [grow cyclic,rotate=180]
  child foreach \a in {0,1} {
    child foreach \b in {0,1} {
      child foreach \c in {0,1} {
        child foreach \d in {0,1} {
          child foreach \e in {0,1} {
            child foreach \leafcolor in {red, orange}
              { edge from parent [color=\leafcolor] }
          }
        }
      }
    } edge from parent
  };
  \node[svertex,label={90:$y$}] (y) at (0,1) {};
  \draw (x) -- (y);
  \foreach \x in {0.167,0.333,...,0.833} {
    \node[fill,color=gray,circle,inner sep=0.75pt] at (0,\x) {};
  }
  \end{scope}
  \end{tikzpicture}
  \caption{A KC-transformation with respect to $x,y$.}
  \label{fig:KC}
\end{figure}
In this section we introduce the main tool that will be used to prove \cref{thm:trees}.
Let $T$ be a tree and let $x,y$ be two distinct vertices of $T$.  We say that the path between $x$ and $y$ is \defn{bare} if for every vertex $v\ne x,y$ on that path, $d_T(v)=2$.
Suppose $x,y$ are such that the unique path $P$ in $T$ between them is bare, and let $z$ be the neighbour of $y$ in that path.  For a vertex $v$, denote by $N(v)$ the neighbours of $v$ in $T$.  The \defn{KC-transformation} $\KC(T,x,y)$ of $T$ with respect to $x,y$ is the tree obtained from $T$ by deleting every edge between $y$ and $N(y)\sm z$  and adding the edges between
$x$ and $N(y)\sm z$ instead
(see \cref{fig:KC}).
Note that $\KC(T,x,y)\simeq \KC(T,y,x)$, so if we only care about unlabelled trees, we may simply write $\KC(T,P)$, for a bare path $P$ in $T$.
The term ``KC-transformation'' was coined by Bollob\'as and Tyomkyn~\cite{BT12} after Kelmans, who defined a similar operation on graphs~\cite{Kel81}, and Csikv\'ari, who defined it in this form~\cite{Csi10} under the name \term{generalized tree shift} (GTS).

A nice property of KC-transformations, first observed by Csikv\'ari~\cite{Csi10}, is that they induce a graded poset on the set of unlabelled trees of a given order, which is graded by the number of leaves.  In particular, this means that in that poset, the path is the unique minimum (say) and the star is the unique maximum.  Note that if $P$ contains a leaf then $\KC(T,P)\simeq T$, and otherwise $\KC(T,P)$ has one more leaf than $T$.
In the latter case, we say that the transformation is \defn{proper}.

\begin{figure}
  \captionsetup{width=0.879\textwidth,font=small}
  \centering
  \begin{tikzpicture}[
    svertex/.style={fill,circle,inner sep=1.5pt},
    scale=1.5,
  ]
  \begin{scope}
  \node[svertex,label={270:$v$}] (v) at (0,0) {} [grow cyclic,rotate=180]
  child foreach \a in {0,1} {
    child foreach \b in {0,1} {
      child foreach \c in {0,1} {
        child foreach \d in {0,1} {
          child foreach \e in {0,1} {
            child foreach \leafcolor in {red, orange}
              { edge from parent [color=\leafcolor] }
          }
        }
      }
    } edge from parent
  };
  \node[svertex] (_v) at (0,0) {} [grow cyclic,rotate=0]
  child foreach \a in {0,1} {
    child foreach \b in {0,1} {
      child foreach \c in {0,1} {
        child foreach \d in {0,1} {
          child foreach \e in {0,1} {
            child foreach \leafcolor in {red, orange}
              { edge from parent [color=\leafcolor] }
          }
        }
      }
    } edge from parent
  };
  \end{scope}
  \begin{scope}[xshift=2.5cm]
  \node (to) at (0,0) {$\implies$};
  \end{scope}
  \begin{scope}[xshift=5cm]
  \coordinate (_y) at (0,0) {} [grow cyclic,rotate=0]
  child[dotted,black!50!white] foreach \a in {0,1} {
    child foreach \b in {0,1} {
      child[solid,black] foreach \c in {0,1} {
        child foreach \d in {0,1} {
          child foreach \e in {0,1} {
            child foreach \leafcolor in {red, orange}
              { edge from parent [color=\leafcolor] }
          }
        }
      }
    } edge from parent
  };
  \node[svertex,black!25!white,label={270:$v$}] (v) at (0,0) {} [grow cyclic,rotate=180]
  child[dotted,black!50!white] foreach \a in {0,1} {
    child foreach \b in {0,1} {
      child[solid,black] foreach \c in {0,1} {
        child foreach \d in {0,1} {
          child foreach \e in {0,1} {
            child foreach \leafcolor in {red, orange}
              { edge from parent [color=\leafcolor] }
          }
        }
      }
    } edge from parent
  };
  \end{scope}
  \end{tikzpicture}
  \caption{A shattering at $v$.}
  \label{fig:shattering}
\end{figure}
Here is the plan for how to prove \cref{thm:trees}.
For a tree $T$ and a vertex $v$, denote by $T\star v$ the forest obtained from $T$ by \defn{shattering} $T$ at $v$, that is, by removing from $T$ the set $\{v\}\cup N(v)$ (see \cref{fig:shattering}).
Denote by $\kappa_v(T)$ the multiset of orders of trees in the forest $T\star v$, and by $\kappa(T)$ the sum of $\kappa_v(T)$ for all vertices $v$ in $T$.
Note that for trees with up to $3$ vertices, \cref{thm:trees} is trivial; we proceed by induction.
By the induction hypothesis,
\begin{equation}\label{eq:tree:induction}
  \egic(T) = \frac{1}{n}
  \sum_{v\in V(T)}\left(1+\sum_{S\in T\star v}\egic(S)\right)
  \ge 1 + \frac{1}{n}\sum_{v\in V(T)} \sum_{k\in\kappa_v(T)} \alpha_k
  = 1 + \frac{1}{n} \sum_{k\in \kappa(T)} \alpha_k.
\end{equation}
Therefore, it makes sense to study the quantities $\nu_v(T)=\sum_{k\in\kappa_v(T)}\alpha_k$ and $\nu(T)=\sum_{k\in\kappa(T)}\alpha_k$.
In fact, it would suffice to show that for any tree $T$ on $n$ vertices $\nu(T)\ge \nu(P_n)$, since by \eqref{eq:path} and \eqref{eq:tree:induction} we would obtain
\begin{equation*}
  \egic(T) \ge 1 + \frac{1}{n}\nu(T) \ge 1 + \frac{1}{n}\nu(P_n)
  = \egic(P_n).
\end{equation*}
We therefore reduced our problem to proving the following theorem about KC-transformations.
\begin{theorem}\label{thm:KC:nu}
  If $T$ is a tree and $P$ is a bare path in $T$ then $\nu(\KC(T,P))\ge\nu(T)$.
\end{theorem}
It would have been nice if for every $v\in V(T)$ we would have had $\nu_v(\KC(T,P))\ge\nu_v(T)$; unfortunately, this is not true in general.  However, the following statement
would suffice.
\begin{theorem}\label{thm:KC:core}
  Let $T$ be a tree and let $x\ne y$ be two vertices with the path between them being bare.
  Denote $T'=\KC(T,x,y)$.
  Let $A$ be the set of vertices $v\ne x$ in $T$ for which every path between $v$ and $y$ passes via $x$, and similarly, let $B$ be the set of vertices $v\ne y$ in $T$ for which every path between $v$ and $x$ passes via $y$.  Let $P$ be the set of vertices on the bare path between $x$ and $y$, so $A\cup B\cup P$ is a partition of $V(T)$.  Then
  \begin{enumerate}
    \item For $v\in A\cup B$ we have that $\nu_v(T')\ge \nu_v(T)$.
    \item $\sum_{v\in P} \nu_v(T') \ge \sum_{v\in P}\nu_v(T)$.
  \end{enumerate}
\end{theorem}

\begin{proof}\
  \begin{enumerate}
    \item
    It suffices to prove the claim for $v\in A$.
    First note that there exists a unique tree $S_v$ in $T\star v$ which is not fully contained in $A$, and the rest of the trees are retained in the KC-transformation.
    The set of trees in $T'\star v$ which are not fully contained in $A$ may be different from $S_v$, but they are on the same vertex set, so the result follows from subadditivity of $\alpha_n$ (\cref{lem:monotone:subadd}).
    \item
    Write $|A|=a$, $|B|=b$ and $|P|=\ell+1$.
    Let $A_1,\ldots,A_s$ be the trees of $T\star x$ which are fully contained in $A$, and denote $a_i=|A_i|$.
    Let $B_1,\ldots,B_t$ be the trees of $T\star y$ which are fully contained in $B$, and denote $b_i=|B_i|$.
    Let $\alpha_A=\sum_{i=1}^s\alpha_{a_i}$, $\alpha^+_A=\sum_{i=1}^s\alpha_{1+a_i}$, $\alpha_B=\sum_{i=1}^t\alpha_{b_i}$ and $\alpha^+_B=\sum_{i=1}^t\alpha_{1+b_i}$.
    Denote the vertices of $P$ by $x=u_0,u_1,\ldots,u_\ell$.  The following table summarises the values of $\nu$ in $T,T'$ along vertices of $P$, in the case where $\ell\ge 3$ (similar tables can be made for the cases $\ell=1,2$).

    \begin{center}
      {\renewcommand{\arraystretch}{1.2}
      \begin{tabular}{@{}lll@{}}
      \toprule
        & $\nu_{u_j}(T)$ & $\nu_{u_j}(T')$ \\\midrule
        $j=0$
        & $\alpha_A + \alpha_{b+\ell-1}$
        & $\alpha_A + \alpha_B + \alpha_{\ell-1}$ \\
        $j=1$
        & $\alpha^+_A + \alpha_{b+\ell-2}$
        & $\alpha^+_A + \alpha^+_B + \alpha_{\ell-2}$ \\
        $2\le j\le \ell-2$
        & $\alpha_{a+j-1} + \alpha_{b+\ell-j-1}$
        & $\alpha_{a+b+j-1} + \alpha_{\ell-j-1}$ \\
        $j=\ell-1$
        & $\alpha_{a+\ell-2} + \alpha^+_B$
        & $\alpha_{a+b+\ell-2}$\\
        $j=\ell$
        & $\alpha_{a+\ell-1} + \alpha_B$
        & $\alpha_{a+b+\ell-1}$\\
        \bottomrule
      \end{tabular}}
    \end{center}

    It follows (for every $\ell\ge 1$) that
    \begin{align*}
    \sum_{v\in P}(\nu_v(T')-\nu_v(T))
    &= \sum_{j=1}^{\ell-1} \left(
    \alpha_{a+b+j} + \alpha_j - \alpha_{a+j} - \alpha_{b+j}
    \right)\\
    &= \xi_{a+b,\ell-1} + \xi_{0,\ell-1} - \xi_{a,\ell-1} - \xi_{b,\ell-1},
    \end{align*}
    which is, by \cref{lem:xi}, nonnegative.\qedhere
  \end{enumerate}
\end{proof}

\subsection{Properties of $\alpha_n$}\label{sec:calc}
We head on to prove \cref{lem:monotone:subadd,lem:xi}.
To simplify presentation, we introduce the following notation.
For a sequence $x_n$ we write
\begin{equation*}
\Delta_h^d x_n=\sum_{k=0}^d(-1)^k\binom{d}{k}x_{n+(d-k)h}
\end{equation*}
to denote its \defn{$d$'th order $h$-forward difference}.  When $h=1$ we omit the subscript, and when $d=1$ we omit the superscript.
The following identities will be useful.  Using \eqref{eq:alpha:explicit},
\begin{equation*}\label{eq:alpha:Delta}
\Delta\alpha_n
= \sum_{i=0}^n \frac{(-2)^i(n+1-i)}{(i+1)!}
-\sum_{i=0}^{n-1} \frac{(-2)^i(n-i)}{(i+1)!}
= \sum_{i=0}^n \frac{(-2)^i}{(i+1)!},
\end{equation*}
and
\begin{equation}\label{eq:alpha:Delta:2}
\Delta^2\alpha_n = \frac{(-2)^{n+1}}{(n+2)!}.
\end{equation}
Note also that since $\alpha_0=0$,
\begin{equation}\label{eq:alpha:sum}
\alpha_n = \sum_{i=0}^{n-1}\Delta\alpha_i,
\end{equation}
and since $\Delta\alpha_0=1=\Delta^2\alpha_{-1}$,
\begin{equation}\label{eq:alpha:Delta:sum}
\Delta\alpha_i = \sum_{j=-1}^{i-1}\Delta^2\alpha_j.
\end{equation}

We proceed by a (rather long) sequence of technical claims,
which will be used in the proofs of \cref{lem:monotone:subadd,lem:xi}.

\begin{lemma}\label{lem:altering}
  Fix $k\ge 0$.
  Let $z_n$ be a real nonnegative decreasing sequence.
  Then $y_{k,n} = (-1)^k\sum_{j=k}^{k+n} (-1)^j z_j$ is nonnegative.
\end{lemma}

\begin{proof}
  Note that for every $\ell\ge 0$ we have that
  \begin{equation*}
  (-1)^{k+2\ell} z_{k+2\ell}\qquad\text{and}\qquad
  (-1)^{k+2\ell} z_{k+2\ell} + (-1)^{k+2\ell+1} z_{k+2\ell+1}
  \end{equation*}
  are both nonnegative if $k$ is even, and both nonpositive otherwise.  The claim easily follows.
\end{proof}

\begin{lemma}\label{lem:altering:subadd}
  Let $z_n$ be a real nonnegative decreasing sequence, which is convex for $n\ge 1$.
  Then, the sequence $x_n=\sum_{i=0}^{n-1}\sum_{j=0}^i (-1)^j z_j$ is monotone increasing and subadditive.
\end{lemma}

\begin{proof}
  For $i,k\ge 0$, write $y_{k,i} = (-1)^k\sum_{j=k}^{k+i} (-1)^j z_j$.
  By \cref{lem:altering} $y_{k,i}\ge 0$, hence $\Delta x_n=y_{0,n}\ge 0$, and $x_n$ is monotone increasing.  Fix $m\ge 1$ and write $b_i=y_{i,m-1}$.  We have that
  \begin{align*}
  a_n := x_{m+n} - x_m - x_n
  &= \sum_{i=0}^{m+n-1} y_{0,i} - \sum_{i=0}^{m-1} y_{0,i}
  - \sum_{i=0}^{n-1} y_{0,i}\\
  &= \sum_{i=0}^{n-1} (y_{0,m+i}-y_{0,i})
  = \sum_{i=0}^{n-1} (-1)^{i+1} y_{i+1,m-1}
  = \sum_{i=1}^n (-1)^i b_i.
  \end{align*}
  Moreover,
  \begin{align*}
  \Delta b_i
  &= (-1)^{i+1}\sum_{j=i+1}^{i+m}(-1)^j z_j - (-1)^i\sum_{j=i}^{i+m-1}(-1)^j z_j\\
  &= (-1)^{i+1}\sum_{j=i}^{i+m-1}(-1)^{j+1} z_{j+1}
  + (-1)^{i+1}\sum_{j=i}^{i+m-1}(-1)^j z_j\\
  &= (-1)^{i+1}\sum_{j=i}^{i+m-1}(-1)^j (-\Delta z_j).
  \end{align*}
  Now, $-\Delta z_j$ is nonnegative, and for $j\ge 1$ it is also decreasing (since $\Delta^2 z_j\ge 0$), thus for $i\ge 1$, by \cref{lem:altering}, $\Delta b_i\le 0$.  Therefore, $b_i$ is nonnegative and decreasing (for $i\ge 1$), hence by \cref{lem:altering}, $a_n\le 0$ for every $n\ge 0$, and thus $x_n$ is subadditive.
\end{proof}

Define \[\beta_n=(-1)^n\Delta^2\alpha_{n-1}.\]
\begin{claim}\label{cl:beta}
  $\beta_n$ is nonnegative and decreasing, and convex for $n\ge 1$.
\end{claim}

\begin{proof}
  From \eqref{eq:alpha:Delta:2} we know that $\beta_n=2^n/(n+1)!>0$.  Moreover,
  \begin{equation*}
  \Delta\beta_n = \frac{2^{n+1}}{(n+2)!}-\frac{2^n}{(n+1)!}
  = \frac{2^n}{(n+1)!}\left( \frac{2}{n+2} - 1\right) \le 0,
  \end{equation*}
  and, for $n\ge 1$,
  \begin{align*}
  \Delta^2\beta_n
  &= \frac{2^{n+2}}{(n+3)!}-\frac{2\cdot 2^{n+1}}{(n+2)!}+\frac{2^n}{(n+1)!}\\
  &= \frac{2^n}{(n+1)!} \left(
  \frac{4}{(n+2)(n+3)} - \frac{4}{n+2} + 1
  \right) \ge 0.\qedhere
  \end{align*}
\end{proof}

We are now ready to prove \cref{lem:monotone:subadd}.
\begin{proof}[Proof of \cref{lem:monotone:subadd}]
  From \eqref{eq:alpha:sum} and \eqref{eq:alpha:Delta:sum} it follows that
  \begin{equation*}
  \alpha_n
  = \sum_{i=0}^{n-1}\Delta\alpha_i
  = \sum_{i=0}^{n-1}\sum_{j=-1}^{i-1}\Delta^2\alpha_j
  = \sum_{i=0}^{n-1}\sum_{j=0}^{i}(-1)^j \beta_j,
  \end{equation*}
  and the result follows from \cref{lem:altering:subadd,cl:beta}.
\end{proof}

Define \[\gamma_n=(-1)^{n+1}\Delta\Delta_2\alpha_n.\]
\begin{claim}\label{cl:gamma}
  $\gamma_n$ is nonnegative and decreasing, and convex for $n\ge 1$.
\end{claim}

\begin{proof}
  Note that (using \eqref{eq:alpha:Delta:2})
  \begin{align*}
  \gamma_n
  &= (-1)^{n+1}\left(\Delta_2\alpha_{n+1}-\Delta_2\alpha_n\right)\\
  &= (-1)^{n+1}\left(\alpha_{n+3}-\alpha_{n+1}-\alpha_{n+2}+\alpha_n\right)\\
  &= (-1)^{n+1}\left(\Delta\alpha_{n+2}-\Delta\alpha_n\right)\\
  &= (-1)^{n+1}\left((\Delta\alpha_{n+2}-\Delta\alpha_{n+1})
    +(\Delta\alpha_{n+1}-\Delta\alpha_n)\right)\\
  &= (-1)^{n+1}\left(\Delta^2\alpha_{n+1}
    +\Delta^2\alpha_n\right)\\
  &= (-1)^{n+1}\left(
  \frac{(-2)^{n+2}}{(n+3)!} + \frac{(-2)^{n+1}}{(n+2)!}
  \right)\\
  &= \frac{2^{n+1}}{(n+2)!}\left(
  \frac{-2}{n+3}+1
  \right) > 0.
  \end{align*}
  Moreover,
  \begin{align*}
  \Delta\gamma_n
  &= (-1)^{n+2}\left(\Delta_2\alpha_{n+2}-\Delta_2\alpha_{n+1}\right)
  -(-1)^{n+1}\left(\Delta_2\alpha_{n+1}-\Delta_2\alpha_n\right)\\
  &= (-1)^n\left(\Delta_2\alpha_{n+2}-\Delta_2\alpha_n\right)\\
  &= (-1)^n\left(\Delta\alpha_{n+3}+\Delta\alpha_{n+2}-\Delta\alpha_{n+1}-\Delta\alpha_n\right)\\
  &= (-1)^n\left((\Delta\alpha_{n+3}-\Delta\alpha_{n+2})
    +2(\Delta\alpha_{n+2}-\Delta\alpha_{n+1})
    +(\Delta\alpha_{n+1}-\Delta\alpha_n)\right)\\
  &= (-1)^n\left(\Delta^2\alpha_{n+2}
    +2\Delta^2\alpha_{n+1}
    +\Delta^2\alpha_n\right)\\
  &= (-1)^n\left( \frac{(-2)^{n+3}}{(n+4)!} + 2\cdot\frac{(-2)^{n+2}}{(n+3)!}
  +\frac{(-2)^{n+1}}{(n+2)!}\right)\\
  &= \frac{2^{n+1}}{(n+2)!}\left(-\frac{4}{(n+3)(n+4)} + \frac{4}{n+3} - 1\right) \le 0,
  \end{align*}
  so $\gamma_n$ is decreasing.  Finally,
  \begin{align*}
  \Delta^2\gamma_n
  &= \gamma_{n+2}-2\gamma_{n+1}+\gamma_n\\
  &= (-1)^{n+1}\left(
  \Delta_2\alpha_{n+3} + \Delta_2\alpha_{n+2}
  - \Delta_2\alpha_{n+1} - \Delta_2\alpha_n
  \right)\\
  &= (-1)^{n+1}\left(
  \alpha_{n+5}+\alpha_{n+4}
  -2\alpha_{n+3}-2\alpha_{n+2}
  +\alpha_{n+1}+\alpha_n
  \right)\\
  &= (-1)^{n+1}\left(
  \Delta\alpha_{n+4}+2\Delta\alpha_{n+3}
  -2\Delta\alpha_{n+1}-\Delta\alpha_n
  \right)\\
  &= (-1)^{n+1}\left(
  (\Delta\alpha_{n+4}-\Delta\alpha_{n+3})
  +3(\Delta\alpha_{n+3}-\Delta\alpha_{n+2})\right.\\
  &\phantom{= (-1)^{n+1}(}\left.
  +3(\Delta\alpha_{n+2}-\Delta\alpha_{n+1})
  +(\Delta\alpha_{n+1}-\Delta\alpha_n)
  \right)\\
  &= (-1)^{n+1}\left(
  \Delta^2\alpha_{n+3}+3\Delta^2\alpha_{n+2}
  +3\Delta^2\alpha_{n+1}+\Delta^2\alpha_n
  \right)\\
  &= (-1)^{n+1}\left(
  \frac{(-2)^{n+4}}{(n+5)!}
  + 3\cdot\frac{(-2)^{n+3}}{(n+4)!}
  + 3\cdot\frac{(-2)^{n+2}}{(n+3)!}
  + \frac{(-2)^{n+1}}{(n+2)!}
  \right)\\
  &= \frac{2^{n+1}}{(n+2)!}\left(
  -\frac{8}{(n+5)(n+4)(n+3)}
  +\frac{12}{(n+4)(n+3)}
  -\frac{6}{(n+3)} + 1
  \right),
  \end{align*}
  which is nonnegative for $n\ge 1$.
\end{proof}

For $a,b\ge 1$, define \[\eta^{a,b}_n=(-1)^{n+1}\Delta_a\Delta_b\alpha_n.\]
\begin{claim}\label{cl:eta:m:2}
  For every $a\ge 1$, $\eta^{a,2}_n$ is nonnegative.
\end{claim}

\begin{proof}
  Note that
  \begin{align*}
  \eta^{a,2}_n
  &= (-1)^{n+1}\left(\Delta_2\alpha_{a+n}-\Delta_2\alpha_n\right)\\
  &= (-1)^{n+1}\sum_{j=n}^{a+n-1}\Delta\Delta_2\alpha_j
  = (-1)^n \sum_{j=n}^{a+n-1}(-1)^j\gamma_j,
  \end{align*}
  which is, by \cref{lem:altering,cl:gamma}, nonnegative.
\end{proof}

\begin{claim}\label{cl:eta:m:1}
  For every $a\ge 1$, $\eta^{a,1}_n$ is nonnegative and decreasing.
\end{claim}

\begin{proof}
  Note that
  \begin{align*}
  \eta^{a,1}_n
  &= (-1)^{n+1}\left(\Delta\alpha_{a+n}-\Delta\alpha_n\right)\\
  &= (-1)^{n+1} \sum_{j=n}^{a+n-1}\Delta^2\alpha_j
  = (-1)^n \sum_{j=n}^{a+n-1}(-1)^j\Delta^2\beta_{j+1},
  \end{align*}
  which is, by \cref{lem:altering,cl:beta}, nonnegative.
  Moreover,
  \begin{align*}
  \Delta\eta^{a,1}_n
  &= (-1)^n\left(
  \Delta\alpha_{a+n+1}-\Delta\alpha_{n+1} + \Delta\alpha_{a+n} - \Delta\alpha_n
  \right)\\
  &= (-1)^n\left( \Delta_2\alpha_{a+n} - \Delta_2\alpha_n \right)
  = (-1)^n\Delta_a\Delta_2\alpha_n = -\eta^{a,2}_n,
  \end{align*}
  which is, by \cref{cl:eta:m:2}, nonpositive, hence $\eta^{a,1}_n$ is decreasing.
\end{proof}

Define
\[\psi^b_n=(-1)^{n+1}(\Delta\Delta_b\alpha_{n+1}+\Delta\Delta_b\alpha_n).\]

\begin{claim}\label{cl:psi}
  For every $b\ge 1$,
  $\psi^b_n$ is nonnegative, and decreasing for $n\ge 1$.
\end{claim}

\begin{proof}
  Note that
  \begin{equation*}
  \psi^b_n
  = (-1)^{n+1} \left( \Delta_b\alpha_{n+2} - \Delta_b\alpha_n \right)
  = (-1)^{n+1}\Delta_2\Delta_b\alpha_n
  = (-1)^{n+1}\Delta_b\Delta_2\alpha_n
  = \eta^{b,2}_n,
  \end{equation*}
  which is, by \cref{cl:eta:m:2}, nonnegative.
  Moreover,
  \begin{align*}
  \Delta\psi^b_n
  &= \eta^{b,2}_{n+1} - \eta^{b,2}_n\\
  &= (-1)^{n+1}\sum_{j=n+1}^{b+n}(-1)^j\gamma_j
  -(-1)^n    \sum_{j=n}^{b+n-1}(-1)^j\gamma_j\\
  &= (-1)^{n+1}\sum_{j=n}^{b+n-1}
  \left( (-1)^{j+1}\gamma_{j+1} + (-1)^j\gamma_j \right)\\
  &= (-1)^{n+1}\sum_{j=n}^{b+n-1}
  (-1)^{j+1}\Delta\gamma_j.
  \end{align*}
  By \cref{cl:gamma}, the sequence $-\Delta\gamma_n$ is nonnegative, and decreasing for $n\ge 1$.  Therefore, by \cref{lem:altering}, $\Delta\psi^b_n$ is nonpositive, thus $\psi^b_n$ is decreasing (for $n\ge 1$).
\end{proof}

\begin{claim}\label{cl:eta}
  For every $a,b\ge 1$, $\eta^{a,b}_n$ is nonnegative, and decreasing for $n\ge 1$.
\end{claim}

\begin{proof}
  Note that
  \begin{align*}
  \eta^{a,b}_n
  &= (-1)^{n+1} (\Delta_b\alpha_{a+n}-\Delta_b\alpha_n)\\
  &= (-1)^{n+1} \sum_{j=n}^{a+n-1} \Delta_b\Delta\alpha_j\\
  &= (-1)^n \sum_{j=n}^{a+n-1} (-1)^j\eta^{b,1}_j,
  \end{align*}
  which is, by \cref{lem:altering,cl:eta:m:1}, nonnegative.
  Moreover,
  \begin{align*}
  \Delta\eta^{a,b}_n
  &= (-1)^{n+2} (\Delta_b\alpha_{a+n+1}-\Delta_b\alpha_{n+1})
  -(-1)^{n+1} (\Delta_b\alpha_{a+n}-\Delta_b\alpha_n)\\
  &= (-1)^n \left(
  \Delta_b\alpha_{a+n+1} - \Delta_b\alpha_{n+1}
  + \Delta_b\alpha_{a+n} - \Delta_b\alpha_n
  \right)\\
  &= (-1)^n \sum_{j=n}^{a+n-1} \left(
  \Delta\Delta_b\alpha_{j+1} + \Delta\Delta_b\alpha_j
  \right)\\
  &= (-1)^{n+1} \sum_{j=n}^{a+n-1} (-1)^{j}\psi^b_n,
  \end{align*}
  which is, for $n\ge 1$, by \cref{lem:altering,cl:psi}, nonpositive, hence $\eta^{a,b}_n$ is decreasing (for $n\ge 1$).
\end{proof}

We are now ready to prove \cref{lem:xi}.
\begin{proof}[Proof of \cref{lem:xi}]
  Note that
  \begin{equation*}
  \xi_{a+b,\ell} + \xi_{0,\ell} - \xi_{a,\ell} - \xi_{b,\ell}
  = \sum_{j=1}^{\ell} \left(\alpha_{a+b+j} + \alpha_j - \alpha_{a+j} - \alpha_{b+j}\right)
  = -\sum_{j=1}^{\ell} (-1)^j \eta^{a,b}_j,
  \end{equation*}
  which is, by \cref{lem:altering,cl:eta}, nonnegative.
\end{proof}

\section{Concluding remarks and open questions}\label{sec:open}

\paragraph*{Non locally tree-like graph sequences}
Our local limit approach does not assume that the converging sequence is locally tree-like.
However, the differential equation tool fails if short cycles appear in a typical local view.
As it seems, to date, there is no general tool to handle these cases, and indeed, even the asymptotic behaviour of the random greedy MIS algorithm on $d$-dimensional tori (for $d\ge 2$) remains unknown.

\paragraph*{Better local rules}
The random greedy algorithm presented here follows a straightforward local rule.
More complicated local rules may yield, in some cases, larger maximal independent sets; for example, the initial random ordering may ``favour'' low degree vertices.
It would be nice to adapt our framework, or at least some of its components, to other settings.
For \emph{adaptive} ``better'' local algorithms, we refer the reader to \cites{Wor95,Wor03}.

\paragraph*{The second colour}
In this work, we have analysed the output of the random greedy algorithm for producing a maximal independent set.
As already remarked, this is, in fact, the set of vertices in the first colour class in the random greedy colouring algorithm.
It is relatively easy to see that, after slight modifications (in particular, in~\cref{thm:de}), this approach allows us to calculate the asymptotic proportion of the size of the set of vertices in the second colour class (or in the k’th colour class in general, for any fixed k) as well.
Non-asymptotic questions about the expected cardinality of the set of vertices in the second colour class might also be of interest.
For example, is it true that the path has the smallest expected number of vertices in the first two colour classes among all trees of the same order?
It is not hard to see that this statement is not true for the first three colour classes (as three colours suffice to colour the path greedily).

\paragraph*{Monotonicity with respect to KC-transformations}
The expected greedy independence ratio in trees is likely to be monotone with respect to KC-transformations and strictly monotone with respect to proper KC-transformations.
If true, this would imply that the greedy independence ratio in trees achieves its {\em unique} minimum on the path.

~

\begin{acknowledgement}
  The authors wish to express their thanks
  to the organisers of the Joint FUB--TAU Workshop on Graph and Hypergraph Colouring,
  hosted by the Freie Universit\"at Berlin in 2018,
  and to Michal Amir, Lior Gishboliner, Matan Harel, Frank Mousset, Matan Shalev and Yinon Spinka
  for useful discussions and ideas.
  We also thank the anonymous referees for the valuable input and suggestions they provided
  that improved the quality of the paper.
\end{acknowledgement}

\bibliography{library}

\end{document}